\documentclass[letterpaper, 12pt]{amsart}
\usepackage[left=0.8 in, right=0.8 in, top=1 in, bottom=1 in]{geometry}

\usepackage{amsmath,amssymb,amsthm,amscd,enumerate,setspace}

\usepackage[pagebackref]{hyperref}
\hypersetup{colorlinks=true, linkcolor=red, citecolor=blue}

\usepackage[all]{xy}
\usepackage{graphicx}

\usepackage{xcolor}
\usepackage{mdframed}
\newmdenv[backgroundcolor=yellow]{shaded}

\input xypic
\xyoption{all}

\long\def\symbolfootnote[#1]#2{\begingroup%
\def\thefootnote{\fnsymbol{footnote}}\footnote[#1]{#2}\endgroup}
\newtheorem{Theorem}{Theorem}[section]

\newtheorem{Corollary}[Theorem]{Corollary}
\newtheorem{Proposition}[Theorem]{Proposition}

\newtheorem{Conjecture}[Theorem]{Conjecture}

\newtheorem{Remark}[Theorem]{Remark}
\newtheorem{Example}[Theorem]{Example}
\newtheorem{Definition}[Theorem]{Definition}
\newtheorem{Question}[Theorem]{Question}

\newtheorem{Sticky Points}[Theorem]{Sticky Points}

\def\adeg{\mbox{\rm adeg}}

\def\Ass{\mbox{\rm Ass}}

\def\bl{[\![}
\def\br{]\!]}
\def\bdeg{\mbox{\rm bdeg}}

\def\cdeg{\mbox{\rm cdeg}}
\def\cl{\overline}

\def\coker{\mbox{\rm coker}}

\def\deg{\mbox{\rm deg}}
\def\Deg{\mbox{\rm Deg}}

\def\ddeg{\mbox{\rm bideg}}
\def\depth{\mbox{\rm depth}}
\def\ds{\displaystyle}

\def\e{\mathrm{e}}

\def\Ext{\mbox{\rm Ext}}

\def\gr{\mbox{\rm gr}}

\def\H{{\mathrm H}}
\def\hdeg{\mbox{\rm hdeg}}
\def\h{\mbox{\rm ht}}
\def\Hom{\mbox{\rm Hom}}
\def\HP{\mbox{\rm P}}

\def\j{\mbox{\rm j}}
\def\jdeg{\mbox{\rm jdeg}}

\def\ker{\mbox{\rm ker}}

\def\lar{\longrightarrow}
\def\l{{\lambda}}

\def\m{\mathfrak{m}}

\def\n{\mathfrak{n}}

\def\p{{\mathfrak p}}
\def\Proj{\mbox{\rm Proj}}

\def\QQ{\mathbb{Q}}

\def\rar{\rightarrow}

\def\reg{\mbox{\rm reg}}

\def\Spec{\mbox{\rm Spec}}

\def\tr{\mbox{\rm tr}}
\def\tratto{\mbox{\rule{3mm}{.2mm}$\;\!$}}

\def\ZZ{\mathbb{Z}}
\def\RR{\mathbb{R}}

\def\aa{{\mathbf a}}

\def\g2{{\mathbf g}}

\def\xx{{\mathbf x}}

\def\zz{{\mathbf z}}

\def\C{\mathcal{C}}

\def\D{\mathcal{D}}
\def\M{\mathfrak{M}}
\def\P{\mathcal{P}}

\def\rmS{{\mathrm S}}

\begin{document}

\title{\sc Degrees: Vasconcelos Contributions}

\thanks{
%AMS 2020 {\em Mathematics Subject Classification}.
%Primary 13H15;  Secondary 13H10, 13A99.\\
{\bf  Key Words:}  multiplicity, arithmetic degree, jdeg, homological degree, cohomological degree, canonical ideal, canonical degree, bi-canonical degree.\\
L. Ghezzi was partially supported by the Fellowship Leave from the New York City College of Technology-CUNY (Fall 2022-Spring 2023) and by a grant from the City University of New York PSC-CUNY Research Award Program Cycle 53.
}

\author{Laura Ghezzi} \address{Department of Mathematics, New York City College of Technology and the Graduate Center, The City University of New York,
 300 Jay Street, Brooklyn, NY 11201, U.S.A.;
365 Fifth Avenue,
New York, NY 10016, U.S.A.}
 \email{laura.ghezzi58@citytech.cuny.edu}

\author{Jooyoun Hong}
\address{Department of Mathematics, Southern Connecticut State
University, 501 Crescent Street, New Haven, CT 06515-1533, U.S.A.}
\email{hongj2@southernct.edu}

\date{\today}

\maketitle

{\em \small In loving memory of Wolmer Vasconcelos. It has been an honor and a privilege to be his friends and collaborators, and to witness the joy that mathematics brought to him. 
}

\begin{abstract} A degree of a module $M$ is a numerical measure of information carried by $M$. We highlight some of Vasconcelos' outstanding contributions to the theory of degrees, bridging commutative algebra and computational algebra. We present several degrees he introduced and developed, including arithmetic degree, jdeg, homological degree, cohomological degrees, canonical degree and bi-canonical degree. For the canonical and bi-canonical degrees we discuss recent developments motivated by our joint works \cite{blue1, bideg, BGHV}.
%open questions, and we present some new examples.
\end{abstract}
%\tableofcontents

\section{Introduction}\label{survey}

We discuss some of Vasconcelos' outstanding results in the theory of degrees and how they have inspired many researches in commutative algebra and beyond. 
%The contributions are too broad to be contained in this manuscript, which is by no means comprehensive. 
%We focus on some highlights and 
We present (in chronological order when appropriate) several degrees Vasconcelos introduced and developed, such as arithmetic degree, jdeg, homological degree, cohomological degrees, canonical degree and bi-canonical degree. For more details and proofs on the statements of various degrees  we refer the reader to \cite{complexity} and to the referenced literature. For canonical degree and bi-canonical degree we focus on our joint work, informal discussions, open questions that were of particular interest to Vasconcelos, and recent developments. A more comprehensive survey is in \cite{BGHHV} and a full treatment is in our papers \cite{blue1, bideg, BGHV}. 
We refer to \cite{BH} for general background and any unexplained terminology.

\medskip

Let $(R,\m)$ be a Noetherian local ring (or a Noetherian graded algebra) and let
$\mathcal M (R)$ be the category of finitely generated $R$-modules (or the appropriate category
of graded modules). Broadly speaking, a degree function is a numerical function
$d : \mathcal M (R) \mapsto \mathbb{N}$. Extensively studied degree functions include the multiplicity
(classical degree) and the Castelnuovo-Mumford regularity. The main point is that the value $d(M)$ should capture several important elements of the
structure of the $R$-module $M$. A major theme of Vasconcelos' prolific research is the connection between commutative algebra and computational algebra (see for instance \cite{V98book}). Therefore we highlight some issues of particular interest when we look for new degree functions.

\medskip

Cohen-Macaulay structures can be considered as fundamental blocks (see \cite{BH} for a beautiful treatment of the subject). Vasconcelos wrote in \cite{complexity}: ``
We hold the view that Cohen-Macaulay structures are ubiquitous in algebra,
are central to commutative algebra and algebraic geometry and represent unique computational
efficiencies and that without an understanding of the ways it appears in
the structure of non Cohen-Macaulay structures the prospects for successful large scale
symbolic computations are rather dim.'' Hence it is desirable to seek functions that coincide with the classical multiplicity when the $R$-module $M$ is Cohen-Macaulay and to be able to measure the deviation of $M$ from being Cohen-Macaulay.  It is also important to relate a new degree function to other classical invariants of $M$ and to study is its behavior under generic hyperplane sections.

\medskip

A main topic of Vasconcelos' research is the study of the complexity of normalization processes of graded algebras. In this context, let $R$ be a normal, unmixed integral domain and let $A$ be a semistandard graded $R$-algebra (i.e., a finite extension of a standard graded $R$-algebra). Let $\overline A$ be the integral closure of $A$. Estimating the number of steps that general algorithms must take to build $\overline A$ can be viewed as an invariant of A. For several interesting classes of algebras, including affine graded algebras over fields and Rees algebras of ideals and modules, multiplicity dependent bounds were derived for instance in \cite{DV, HUV, PUV, UV, V91}. One of Vasconcelos' broad goals was to define new degree functions that provide extensions of such bounds to general graded algebras and to find new techniques to study the complexity.

\medskip
Now we briefly describe how this manuscript is organized. The first half of the paper focuses on arithmetic degree, jdeg, homological degree, and cohomological degrees.
The definitions, properties and developments of these degrees are in the respective sections. 

\medskip

The {\em arithmetic degree} is addressed in Section \ref{AD}. It is a refinement of the classical multiplicity, but it is defined in a way that collects information about the associated
primes of $M$ in all codimensions. A major contribution of Vasconcelos is a formula for the arithmetic degree that does not use a primary decomposition (Theorem \ref{adW}). Due to his result, software systems such as Macaulay2 \cite{Macaulay2} or CoCoA \cite{CoCoA} can be used to compute the arithmetic degree. 

\medskip

Section \ref{jdegsection} is dedicated to {\em 
$jdeg$}, introduced by Pham and Vasconcelos as a generalization of $\j$-multiplicity. One important point is that the $\jdeg$ provides estimates for the complexity of normalization (Theorem \ref{jdW}). It also extends classical bounds.

\medskip 

In Section \ref{HDsection} we discuss the {\em homological degree}, which was introduced by Vasconcelos as an extension of
the usual notion of multiplicity, aimed to provide a numerical signature for
a module $M$ when $M$ is not necessarily Cohen-Macaulay. The difference between the homological degree of $M$ and its multiplicity is often called a Cohen-Macaulay deficiency in the literature.  
For example, if $M$ is generalized Cohen-Macaulay, this difference is the St\"{u}ckrad-Vogel invariant of $M$. The homological degree is defined recursively on the dimension of the module, which allows,
in principle, for its calculation by symbolic computer packages. It can be used to bound the Euler characteristic of parameter ideals (Theorem \ref{HDbound}).
%Futhermore, if $M$ is Buchsbaum, then the homological degree can be expressed in terms of multiplicities. 
The homological degree is used to define the {\em homological torsion}, which plays an important role in the study of Hilbert coefficients. The homological torsion can be used to bound the first Hilbert coefficient $\e_{1}(Q, M)$ of a module $M$ relative to an ideal generated by a system of parameters $Q$ for $M$ (Theorem \ref{HTbound}). The coefficient $\e_{1}(Q, M)$ is often called the Chern number, and it has been extensively studied by Vasconcelos.

\medskip

In Section \ref{CDsection} we discuss {\em cohomological degrees}, introduced by Vasconcelos. 
The notion of cohomological degree generalizes the homological degree and it is the classical multiplicity if $M$ is Cohen-Macaulay. It provides an estimate for the Castelnuovo-Mumford regularity of a standard graded algebra (Theorem \ref{CDmumford}). The cohomological degree can also be used to estimate the minimal number of generators of an ideal, generalizing and improving many previously known bounds (Theorem \ref{CDbound}). 
We also discuss the {\em bdeg}, a cohomological degree introduced by Vasconcelos' former student Gunston.
\medskip

The second half of the manuscript, Sections \ref{canonicalsection}, \ref{sallysubsection}, \ref{bicanonicalsection}, \ref{comparisonsection}, \ref{changeringsection},
\ref{bicanonicalgeneralizationsection}, deals with more recently introduced degrees, the canonical degree and the bi-canonical degree \cite{blue1, bideg, BGHV}. 
We include topics and developments motivated by our discussions with Vasconcelos. 
%We also present some new examples (in the context of 3-AGL rings), computed with the aid of Macaulay2.
%These are joint papers with Vasconcelos, and therefore we include topics and current developments motivated by our discussions with him. We also present some new examples.

\medskip

The overarching goal here is the stratification of Cohen-Macaulay rings. Several researchers 
have been interested in finding new significant classes of rings which naturally include Gorenstein rings. We introduce invariants that characterize some of these classes and measure the deviation from $R$ being Gorenstein.
Let $R$ be a Cohen-Macaulay local ring that has a canonical ideal $\C$.  A canonical ideal is an ideal isomorphic to a canonical module of $R$. The condition that the canonical module is an ideal is equivalent to $R$ being generically Gorenstein. We refer to \cite[Section 3]{BH} for the background on canonical modules.
The ring $R$ is Gorenstein when $\C$ is isomorphic to $R$, therefore we look at the properties of $\C$ as a way to refine our understanding of $R$.

\medskip

In Section \ref{canonicalsection} we discuss the {\em canonical degree} of $R$, $\cdeg(R)$. It is defined so that its vanishing characterizes the Gorenstein property of $R$ in codimension one. Furthermore, the minimal value of the canonical degree is assumed on the class of {\em almost Gorenstein rings} (Proposition \ref{1dimalmostg}). Almost Gorenstein rings local rings (AGL for short) were introduced in \cite{BF97}, and a rich theory has been developed in \cite{GMP11} and \cite{GTT15}. Numerical semigroup rings, in particular the three-generated ones, provide the set-up for many interesting results and examples, which we summarize at the end of the section. In Section \ref{sallysubsection} we relate the canonical degree to the Sally module, introduced by Vasconcelos. This connection inspired the definition of 2-AGL rings \cite{CGKM}, as a generalization of AGL rings. We propose a further generalization to 3-AGL rings, we give a characterization in terms of canonical degree when the type of the ring is small, we present supporting examples in three-generated semigroup rings (computed with the aid of Macaulay2), and we conclude with some open questions.

 \medskip
 
In Section \ref{bicanonicalsection} we discuss the {\em bi-canonical degree} of $R$, $\ddeg(R)$. Like the canonical degree, the vanishing of the bi-canonical degree characterizes Gorensteinness in codimension one. However, the bi-canonical degree seems more suitable for computations. We recall the important properties and we explore rings with minimal value of bideg. 

\medskip

In Section \ref{comparisonsection} we discuss the Comparison Conjecture. In \cite{bideg} we conjectured that in general, $\ddeg(R) \leq \cdeg(R)$.  We prove that the conjecture holds in three-generated numerical semigroup rings, using the methods of \cite{bideg}, where we had shown a special case. There have been very interesting recent developments  for numerical semigroup rings. Using different methods, Herzog and Kumashiro in \cite{HK22} studied an inequality that turned out to be equivalent to the Comparison Conjecture. Their results imply that the conjecture is true when the type of $R$ is at most three, and they found a counterexample when the type is five. 

\medskip
In Section \ref{changeringsection} we summarize how the canonical degree and the bi-canonical degree behave under change of rings, in particular augmented rings, $(\m:\m)$, and hyperplane sections.

\medskip

Finally, in Section \ref{bicanonicalgeneralizationsection}, we present the main results and open questions from \cite{BGHV}, our last  paper with Vasconcelos. We extended the bi-canonical degree to rings where the canonical module is not necessarily an ideal, and we discussed generalizations to rings without canonical modules but admitting modules sharing some of their properties.

\section{Arithmetic Degree}\label{AD}

We begin this section by recalling the definition of multiplicity, which motivated the definition of the arithmetic degree. Let $R$ be a Noetherian local ring of dimension $d>0$ with maximal ideal $\m$ and let $M$ be a finitely generated $R$-module of dimension $s$.  Let $I$ be an $\m$-primary ideal of $R$. We denote length by $\lambda$. We define the Hilbert function of $M$ with respect to $I$ as  
\[ \H_{I, M}: n \mapsto \l(M/I^{n+1}M). \]
There exists a polynomial ${\ds \HP_{I, M}(X) \in \QQ[X]}$ of degree $s$ such that ${\ds \H_{I, M}(n) =  \HP_{I, M}(n)}$ for all sufficiently large $n$. This polynomial ${\ds \HP_{I, M}(X)}$ is called the Hilbert polynomial of $M$ with respect to $I$ and can be written as  
\[ \HP_{I, M}(n) = \sum_{i=0}^{s} (-1)^{i} \e_{i}(I, M) {{n+s-i}\choose{s-i}}. \]
The integers $\e_{i}(I, M)$ are called the Hilbert coefficients of $M$ with respect to $I$. In particular, $\e_{0}(I, M)$ is called the multiplicity (or degree) of the module $M$ with respect to $I$. When $I=\m$, we write ${\ds \e_{0}(\m, M) = \deg(M)}$.

\medskip

There is a chain of submodules ${\ds 0=M_{0} \subset M_{1} \subset \cdots \subset M_{r-1} \subset M_{r}=M}$ such that ${\ds M_{i}/M_{i-1} \simeq R/\p_{i}}$ for some ${\ds \p_{i} \in \Spec(R)}$.
In this case, we obtain
\[ \deg(M) = \sum_{\tiny \dim(R/\p_{i})=d} \deg(R/\p_{i}). \]
Recall that ${\ds \Ass(M) \subseteq \{\p_{i}\}_{i=1}^{n} }$. %(ComAlgNote 6.1.11)
If ${\ds \p \in \Ass(M)}$ with ${\ds \dim(R/\p)=d}$, then the length ${\ds \l(M_{\p})}$ is the number of times that $\p$ occurs in the chain of submodules. That is, we have
\[ \sum_{\tiny \begin{array}{c} \p \in \Ass(M) \\  \dim(R/\p)=d \end{array}} \l(M_{\p}) \, \deg(R/\p) \leq \deg(M). \]
The arithmetic degree of $M$ was defined in a way that collects information about the associated primes of $M$ in {\em all} codimensions. 

\begin{Definition}{\rm \cite[Definition 1.6]{V98} Let $(R, \m)$ be a Noetherian local ring and $M$ a finitely generated $R$-module. 
  The {\em arithmetic degree} of $M$ is the integer
\[ \adeg(M) = \sum_{\tiny \p \in \Ass(M)}  \l( \Gamma_{\p}(M_{\p})) \, \deg(R/\p), \]
where ${\ds \Gamma_{\p}(M_{\p}) = H^{0}_{\p R_{\p}}(M_{\p}) }$ is the finite length submodule of $M_{\p}$ consisting of all elements of ${\ds M_{\p}}$ with support in ${\ds \p R_{\p}}$. (In some literatures, the arithmetic degree is also denoted by ${\ds \mbox{arith-deg}(M)}$.) 
}\end{Definition}

If all the associated primes of $M$ have the same codimension, then $\adeg(M)$ is equal to $\deg(M)$.  
Both $\deg(M)$ and $\adeg(M)$ can be defined in the same way as above when $M$ is a (graded) finitely generated module over a standard graded algebra ${\ds A=A_{0}[A_{1}]}$ with the Artinian local ring $A_{0}$. In this case, the Hilbert function of $M$ is defined in terms of the homogeneous maximal ideal of $A$. Bayer and Mumford \cite{BM93} gave an extensive survey on ${\ds \adeg(A)}$ and its significance in algebraic geometry, where ${\ds A=k[X_{1}, \ldots, X_{n}]/I}$ is the quotient of a polynomial ring over a field $k$. 

\medskip

Vasconcelos showed how to compute the arithmetic degree of a graded module without using a primary decomposition. Due to his theorem, software systems such as Macaulay2 or CoCoA can be used to compute the arithmetic degree. 

\begin{Theorem}{\rm (\cite[Proposition 1.11]{V98}, \cite[Proposition 5]{V96})}\label{adW}
Let $R=k[X_{1}, \ldots, X_{d}]$ and $M$ a graded $R$-module. Then
\[ \adeg(M) = \sum_{i=0}^{d}  \deg \left( \Ext^{i}_{R} \left( \Ext^{i}_{R}(M, R), R \right)  \right). \] 
\end{Theorem}

One of the main characteristics of a new degree function $ \Deg(M)$ (see Section 5) that we wish to obtain is that it behaves well under generic hyperplane sections. That is, if $A=A_{0}[A_{1}]$ is a standard graded algebra with the Artinian local ring $A_{0}$ and if $h \in A_{1}$ is not contained in any minimal prime of $A$, then ${\ds \Deg(M) \geq \Deg(M/hM) }$. Unfortunately, the arithmetic degree fails to obtain this characteristic.

\begin{Theorem}\label{V98-1-14}{\rm \cite[Theorem 1.14]{V98}}
Let $A=A_{0}[A_{1}]$ be a standard graded algebra with the Artinian local ring $A_{0}$  and $M$ a finitely generated graded $A$-module. Let $h \in A_{1}$ be a regular element on $M$. Then
\[  \adeg(M) \leq \adeg(M/hM). \]
\end{Theorem}

In \cite{V96} Vasconcelos explored relations between the Castelnuovo-Mumford regularity and the arithmetic degree of a standard algebra $A$. 
He also studied comparisons between the arithmetic degree and the reduction number of $A$. He extensively studied the reduction number as a measure of the complexity of $A$. 
One interesting application of the arithmetic degree is that it can be used to make predictions about the outcome of carrying out a Noether normalization on a graded algebra.

\begin{Theorem}{\rm (\cite[Theorem 2.11]{V98}, \cite[Theorem 7]{V96})}
Let $A$ be an affine algebra over an infinite field $k$ and let $k[\zz]$ be a Noether normalization of $A$. Let $M$ be a finitely generated,  graded, and faithful $A$-module. Then every element of $A$ satisfies a monic equation over $k[\zz]$ of degree at most $\adeg(M)$.
\end{Theorem}

Miyazki, Vogel, and Yanagawa \cite{MVY97} presented a refined theory on the arithmetic degree of a homogeneous ideal $I$ of a polynomial ring $S$ over a field. A key point in their work is the description of the arithmetic degree in terms of the Hilbert polynomial of an appropriate $\Ext_S^i(S/I,S)$, 
based on Vasconcelos' ideas \cite{EHV, V96}.

\medskip

\section{$\jdeg$}\label{jdegsection}

The notion of $\j$-multiplicity was introduced and developed by Achilles and Manaresi \cite{AM93}, Flenner, O'Carroll and Vogel \cite{FOV99}, and evolved into a rich
extension of the ordinary multiplicity (see for instance \cite{NU, UV07}).  Let $(R, \m)$ be a Noetherian local ring and let $A$ be a finitely generated graded $R$-algebra. Let $M$ be a finitely generated graded $A$-module. Then $\H^{0}_{\m}(M)$ is a graded submodule of $M$ and is annihilated by a sufficiently large power of $\m$. Therefore,  $\H^{0}_{\m}(M)$ can be considered as a graded module over the Artinian ring ${\ds A/\m^{k}A}$ for some sufficiently large $k$. 

\begin{Definition}{\rm Let $(R, \m)$ be a Noetherian local ring and let $A$ be a finitely generated graded $R$-algebra. Let $M$ be a finitely generated graded $A$-module. Then the {\em $\j$-multiplicity} of $M$ is 
\[  \j_{\m}(M) = \left\{  \begin{array}{ll}
{\ds \deg( \H_{\m}^{0}( M )  ) } \quad &\mbox{if} \;\; {\ds \dim ( \H_{\m}^{0}( M )  ) = \dim_{R} (M)  }, \vspace{0.1 in} \\
0 \quad &\mbox{otherwise.}
\end{array} \right. \]
}\end{Definition}

A special case of  the $\j$-multiplicity can be defined as follows: Let $(R, \m)$ be a Noetherian local ring, $I$ an arbitrary $R$-ideal, and $N$ a finite $R$-module of dimension $d$. Then the associated graded module ${M= \gr_{I}(N) = \bigoplus_{j=0}^{\infty} I^{j}N/I^{j+1}N}$ is a finitely generated graded module over the associated graded ring ${A=\gr_{I}(R) = \bigoplus_{j=0}^{\infty} I^{j}/I^{j+1}}$. In this case, the $\j$-multiplicity of  ${\ds \gr_{I}(N)}$ is sometimes called the $\j$-multiplicity of $I$ on $N$. Polini and Xie \cite{PX13} established the depth of the associated graded ring in terms of the $\j$-multiplicity of an ideal. 

\medskip
Pham and Vasconcelos introduced the $\jdeg$, extending the definition of $\j$-multiplicity by localization. We note that in contrast to other extensions of the classical multiplicity, such as the arithmetic degree, that requires $R$ to be a local ring, $\jdeg$ places no such restrictions on $R$.

\begin{Definition}{\rm \cite[Definition 1.1]{PV10}  Let $R$ be a Noetherian ring and $A$ a standard graded $R$-algebra. Let $M$ be a finitely generated graded $A$-module. For each prime ideal $\p$  of $R$,  we define
\[  \j_{\p}(M) = \left\{  \begin{array}{ll}
{\ds \deg( \H_{\p R_{\p} }^{0}( M_{\p} )  ) } \quad &\mbox{if} \;\; {\ds \dim ( \H_{\p R_{\p}}^{0}( M_{\p} )  ) = \dim_{R_{\p}} (M_{\p})  }, \vspace{0.1 in} \\
0 \quad &\mbox{otherwise.}
\end{array} \right. \]
Then $\jdeg$ of $M$ is the integer
\[ \jdeg(M) = \sum_{\tiny \p \in \Spec(R)} \j_{\p}(M). \]
}\end{Definition}

If $(R, \m)$ is an Artinian local ring, then  $\jdeg(M)$ is equal to the traditional multiplicity $\deg(M)$. One important application is that the $\jdeg$ provides estimates for the complexity of normalization.
 
\begin{Theorem}{\rm \cite[Theorem 2]{PV08}}\label{jdW}
Let $R$ be a Noetherian domain and $A$ a semistandard graded $R$-algebra of finite integral closure $\overline A$. Consider a sequence of distinct integral graded extensions
\[A \subseteq A_0 \rightarrow A_1 \rightarrow A_2 \cdots \rightarrow A_n=\overline A,\]
where the $A_i$ satisfy the $S_2$ condition of Serre. Then $n \leq \jdeg(\overline A/A)$.
\end{Theorem}

In \cite{PV08} Pham and Vasconcelos also explored how the values of  $\jdeg$ relate to other well studied invariants of the algebra. For instance, let $A$ be a semistandard graded domain over a field $k$. If $A$ satisfies $S_2$ and $A \neq \overline A$, then $\jdeg(\overline A/A)=\e_1(A)-\e_1(\overline A)$. Among other interesting results, using $\jdeg$ the authors generalize bounds of \cite{PUV} (see \cite[Theorem 3]{PV08}). In the sequel \cite{PV10} the focus is on $\jdeg$ of blow-up algebras. We denote the minimal number of generators of a module $M$ by $\nu(M)$.

\begin{Theorem}{\rm \cite[Proposition 4.1]{PV10}}
Let ${\ds (R, \m)}$ be a Noetherian local domain of positive dimension and let $M$ be a non-free module of rank r that is free on the punctured spectrum. Let $S$ be the symmetric algebra of $M$. Then
\[ \jdeg(S) = \left\{\begin{array}{ll}
1 + \j_{\m}(S) \quad &\mbox{if} \;\; {\ds  \nu(M) \geq \dim(R) + r, } \vspace{0.1 in} \\
1 \quad &\mbox{otherwise.}
\end{array} \right. \]
\end{Theorem}

A more specific bound of $\jdeg$ of the symmetric algebra with additional conditions on $R$ and $M$ can be found in \cite[Theorem 4.2 and Proposition 4.3]{PV10}. An important property of $\jdeg$ is its invariance property. 

\begin{Theorem}{\rm \cite[Theorem 4.1.1]{P06} } Let $R$ be a reduced universally catenary Noetherian ring and $I$ an ideal containing a regular element. Let $A=R[It]$ be the Rees algebra of $I$ and let $\cl{A}$ be the integral closure of $A$. Let $B$ be a finitely generated graded $A$-algebra defined by a decreasing filtration such that ${\ds A \subsetneq B \subset \cl{A}}$. Then
\[ \j_{\p}(\gr(A)) = \j_{\p}(\gr(B)) \quad \mbox{for every} \;\; \p \in \Spec(R). \]
Moreover, if $B$ satisfies the Serre's condition $(\rmS_{2})$, then ${\ds \jdeg(\gr(A)) = \jdeg(\gr(B))}$. 
\end{Theorem}

%\begin{proof} \colorbox{yellow}{Comment.} Add Proof? Wolmer's note 52pp. This is not in \cite{PV10} but could have been in Thuy's Thesis.
%\end{proof}

\section{Homological Degree}\label{HDsection}

The homological degree was introduced by Vasconcelos as an extension of
the usual notion of multiplicity, aimed to provide a numerical signature for
a module $M$ when $M$ is not necessarily Cohen-Macaulay. It is defined recursively on the dimension of the module, which allows,
in principle, for its calculation by symbolic computer packages.

\begin{Definition}\label{hd}{\rm \cite[Definition 2.8]{V98-1} Let $(R, \m)$ be a Gorenstein local ring of dimension $d$. Let $I$ be an $\m$-primary ideal of $R$. Let $M$ be a finitely generated $R$-module of dimension $s$. The {\em homological degree} of $M$ with respect to $I$ is defined as follows:
\[ \hdeg_{I}(M) = \left\{\begin{array}{ll}
{\ds \l(M)}  \quad & \mbox{if} \;\; {\ds s =0, } \vspace{0.1 in} \\
{\ds \e_{0}(I, M) + \sum_{i=d-s+1}^{d} {{s-1}\choose{i-d+s-1}} \hdeg_{I} \left( \Ext^{i}_{R}(M, R) \right)   }  & \mbox{if} \;\;  {\ds  s > 0. } 
\end{array} \right. \] 
When $I=\m$, we simply write ${\ds \hdeg(M)}$. 
}\end{Definition}

If $(R, \m)$ is a Gorenstein local ring, by the Local Duality Theorem \cite[Theorem 3.5.8]{BH}, for all integers $i$, there exist natural isomorphisms
\[ \Ext^{i}_{R} (M, R) \simeq \Hom_{R}( H^{d-i}_{\m}(M), \, E),\]
where $E=E(k)$ denotes the injective envelope of the residue field $k$.

\medskip
Goto and Ozeki \cite{GO14} introduced an alternative definition of a homological degree. 

\begin{Definition}\label{hd1}{\rm \cite[Definition 1.1]{GO14} Let $(R, \m)$ be a Gorenstein local ring with infinite residue field.  Let $I$ be an $\m$-primary ideal of $R$ and $M$ a finitely generated $R$-module of dimension $s$. 
The {\em homological degree} of $M$ with respect to $I$ is defined as follows:
\[ \hdeg_{I}(M) = \left\{\begin{array}{ll}
{\ds \l(M)}  \quad & \mbox{if} \;\; {\ds s=0 } \vspace{0.1 in} \\
{\ds  \e_{0}(I, M) + \sum_{i=0}^{s-1} {{s-1}\choose{i}} \hdeg_{I} \left( \Hom_{R}(H^{i}_{\m}(M), E) \right)   }  & \mbox{if} \;\;  {\ds  s > 0. } 
\end{array} \right. \] 
}\end{Definition}

By \cite[Fact 2.1]{GO14}, ${\ds \dim\left(  \Hom_{R}(H^{i}_{\m}(M), E)  \right) \leq i}$ for each $i$. 
Therefore the definition of the homological degree is a recursive definition. Moreover, if $\dim(R)=\dim(M)$, then  the definitions \ref{hd} and \ref{hd1} coincide. Note that if $M$ is Cohen-Macaulay, then ${\ds \hdeg_{I}(M) = \e_{0}(I, M)  }$.

\medskip

 In the original definition of the homological degree  \cite[Definition 2.8]{V98-1},  Vasconcelos supposed that $R$ is a graded algebra, $S$ a Gorenstein graded algebra mapping onto $R$ and $M$ is a finitely generated graded module of $R$, and he used ${\ds \Ext_{S}^{i}(M, S)}$ in the definition.
  In \cite{GO14}, the authors supposed that $R$ is $\m$-adically complete. In this case, $R$ is a homomorphic image of a Gorenstein local ring $S$. Thus, by passing to $S$, without loss of generality, they assumed that $R$ is a Gorenstein ring. Throughout this section we may assume that $R$ is a Gorenstein local ring, not only for simplicity, but also because it enables us to use both definitions (\cite[Definition 2.8]{V98-1}, \cite[Definition 1.1]{GO14}) due to the Local Duality Theorem.

\begin{Remark}\label{hdsmall}{\rm The following are the homological degrees of a module $M$ of small dimension. 

\begin{enumerate}[(1)]
\item Suppose $\dim(M)=1$. Then 
\[  \hdeg_{I}(M) = \e_{0}(I, M) + \l \left( H^{0}_{\m}(M)  \right).\]
\item Suppose $\dim(M)=2$. Let ${\ds N= \Hom(H^{1}_{\m}(M), E) }$. Then
\[ \hdeg_{I}(M) = \left\{  \begin{array}{ll}
{\ds \e_{0}(I, M) + \l \left( H^{0}_{\m}(M)  \right) + \l \left( N \right)} \quad & \mbox{if} \;\;  \dim(N) =0,  \vspace{0.1 in} \\ {\ds  \e_{0}(I, M) + \l \left( H^{0}_{\m}(M)  \right) + \e_{0}(I, N) + \l \left( H^{0}_{\m}(N)  \right)}  \quad & \mbox{if} \;\;  \dim(N) =1.  
\end{array} \right. \]
Also, by \cite[Proposition 3.5]{V98}, ${\ds \hdeg(M)= \adeg(M) + \l \left( \Ext^{2}_{R} \left( \Ext^{1}_{R}(M,R), R  \right)  \right) }$. 
\end{enumerate} 
}\end{Remark}

%%--------------------------
%\begin{Remark}{\rm One way to extend the notion of the homological degree to a non-Gorenstein ring would be that we suppose that there exists a canonical module $\omega_{R}$ and that we replace ${\ds \Ext^{i}_{R}(M, R)}$ by ${\ds \Ext^{i}_{R}(M, \omega_{R})}$ in the definition (\ref{hd}).
%}\end{Remark}
%%---------------------------

\medskip

A finitely generated module $M$ is said to be {\em generalized Cohen-Macaulay} if $M_{\p}$ is Cohen-Macaulay for every non-maximal prime ideal $\p$. Recall that $M$ is generalized Cohen-Macaulay if and only if $H^{i}_{\m}(M)$ is finitely generated for all $i < \dim(M)$. By using the Matlis Duality, we obtain the following.

\begin{Proposition} 
Let $(R, \m)$ be a Gorenstein local ring and $I$ an $\m$-primary ideal. 
If $M$ is a generalized Cohen-Macaulay module and $\dim(M)= s$, then
\[ \hdeg_{I}(M) = \e_{0}(I, M) + \sum_{i=0}^{s-1} {{s-1}\choose{i}}\, \l \left( H^{i}_{\m} (M)  \right). \]
\end{Proposition}

The integer  ${\ds \sum_{i=0}^{s-1} {{s-1}\choose{i}}\, \l \left( H^{i}_{\m} (M)  \right)}$ is called the St\"{u}ckrad-Vogel invariant of $M$ and is denoted by ${\ds \mathbb{I}(M)}$. Thus, if $M$ is a generalized Cohen-Macaulay module, then ${\ds \hdeg_{I}(M) = \e_{0}(I, M) + \mathbb{I}(M) }$.  A family of elements ${\ds \xx=x_{1}, \ldots, x_{s}}$ in the maximal ideal $\m$ is said to be a {\em system of parameters} for $M$ if ${\ds \dim(M/(\xx)M) =0}$, where $s= \dim(M)$. We call an $R$-ideal $Q$ a {\em parameter ideal} for $M$ if there exists a system of parameters ${\ds \xx=x_{1}, \ldots, x_{s}}$ for $M$ such that $Q=(\xx)$. 
A finitely generated module $M$ is said to be {\em Buchsbaum}  if ${\ds  \l(M/QM) - \e_{0}(Q, M) = \mathbb{I}(M)  }$ for every parameter ideal $Q$ of $M$. It is well known that a Buchbaum module is a generalized Cohen-Macaulay module. Thus, we obtain the following formula for $\hdeg_{I}(M)$ in terms of multiplicities.

\begin{Proposition}
Let $(R, \m)$ be a Gorenstein local ring and $I$ an $\m$-primary ideal. 
If $M$ is a Buchsbaum module, then for every parameter ideal $Q$ of $M$ we have
\[ \hdeg_{I}(M) = \e_{0}(I, M) + \l(M/QM) - \e_{0}(Q, M). \]
\end{Proposition}

\subsection*{Properties of  homological degrees} Now we observe general properties of the homological degree. Particularly, we are interested in bounds for $\hdeg$, how $\hdeg$ works with short exact sequences, and most importantly how $\hdeg$ behaves under hyperplane sections. 

\begin{Proposition}{\rm \cite[Proposition 7.6]{Chern1}} Let $(R, \m)$ be  a Gorenstein local ring and $I$ an $\m$-primary ideal such that $\m^{r} \subset I$. Let $M$ be a $R$-module of dimension $s$. Then
\[ \hdeg_{I}(M) \leq r^{s} \, \deg(M) + r^{s-1} \, \left( \hdeg(M) - \deg(M)  \right). \]
\end{Proposition}

\begin{Proposition}\label{V98-3-18}{\rm (\cite[Proposition 3.18]{V98}, \cite[Proposition 2.84]{V2}, \cite[Lemma 2.4]{GO14}) }
Let $(R, \m)$ be  a Gorenstein local ring and $I$ an $\m$-primary ideal.  Let ${\ds 0 \rar L \rar M \rar N \rar 0}$ be an exact sequence of finitely generated $R$-modules. Then
\begin{enumerate}[{\rm (1)}]
\item If ${\ds \l(L) < \infty}$, then ${\ds \hdeg_{I}(M) = \hdeg_{I}(L) + \hdeg_{I}(N)}$.
\item If ${\ds \l(N) < \infty}$, then ${\ds \hdeg_{I}(M) \leq \hdeg_{I}(L) + \hdeg_{I}(N)}$.
\item If ${\ds \l(N) < \infty}$ and ${\ds \dim(M) =s}$, then ${\ds \hdeg_{I}(L) \leq \hdeg_{I}(M) + (s-1) \hdeg_{I}(N)}$.
%\item If ${\ds \l(N) < \infty}$ and ${\ds \depth(M) \geq 2}$, then ${\ds \hdeg_{I}(L) = \hdeg_{I}(M) + \hdeg_{I}(N)}$.
\end{enumerate}
\end{Proposition}

%%-------------------------
%\begin{proof} (1) and (2), (3) are proved in \cite[Proposition 3.18]{V98} and \cite[Lemma 2.4]{GO14}. 
%
%\medskip
%
%\colorbox{yellow}{Comment.}  (4) is in Wolmer's note (Proposition 5.20 in page 71). Should the proof be included here?
%\end{proof}
%%--------------------------

If $R$ is a standard Gorenstein graded algebra and  $M$ is a finitely generated graded $R$-module of depth at least $1$, then Vasconcelos proved that ${\ds \hdeg(M/hM) \leq \hdeg(M)}$ for any generic hyperplane section $h$  on $M$ \cite[Theorem 2.13]{V98-1}. Goto and Ozeki extended it to a local ring case. We denote the Rees algebra of an ideal $I$ by ${\ds \mathcal{R}(I)= \bigoplus_{n=0}^{\infty} I^{n} t^{n} }$. We write ${\ds \mathcal{R}_{+}(I) = \bigoplus_{n=1}^{\infty} I^{n} t^{n} }$ and set
\[ \Proj \left( \mathcal{R}(I) \right)  = \left\{ \p \mid \p \; \mbox{is a graded prime ideal of $\mathcal{R}(I)$ such that $\mathcal{R}_{+}(I)  \not \subseteq \p$ }   \right\}. \]
Let ${\ds f: I \rar \mathcal{R}(I)}$ be a map given by ${\ds f(x)=xt}$. 

\begin{Theorem}\label{GO14-2-6}{\rm \cite[Lemma 2.6]{GO14} } 
Let $(R, \m)$ be a Gorenstein local ring and $I$ an $\m$-primary ideal. Let $M$ be a finitely generated $R$-module. There exists a finite subset ${\ds \mathcal{F} \subseteq \Proj \left( \mathcal{R}(I) \right)  }$ such that  
\begin{enumerate}[{\rm (1)}]
\item every element of ${H= I \setminus \bigcup_{\p \in \mathcal{F}} \left( f^{-1}(\p) + \m I  \right) }$ is superficial for $M$ with respect to $I$, and
\item ${\ds \hdeg_{I}(M/a M) \leq \hdeg_{I}(M)}$ for each $a \in H$.
\end{enumerate}
\end{Theorem}

\begin{Definition}{\rm
Let $(R, \m)$ be a Gorenstein local ring and $I$ an $\m$-primary ideal. Let $M$ be a finitely generated $R$-module. 
An element $a \in R$ (or a sequence $\aa=a_{1}, \ldots, a_{s}$ of elements of $R$) is said to be  {\em superficial} relative to $\hdeg_{I}(M)$ if $a$ (or the sequence $\aa$) is superficial for $M$ with respect to $I$ and ${\ds \hdeg_{I}(M/a M) \leq \hdeg_{I}(M) }$ (or ${\ds \hdeg_{I}(M/\aa M) \leq \hdeg_{I}(M)}$ respectively). 
}\end{Definition}

\begin{Corollary}\label{Chern1-7-2}{\rm \cite[Theorem 7.2]{Chern1}} 
Let $(R, \m)$ be a Gorenstein local ring and $I$ an $\m$-primary ideal. Let $M$ be a finitely generated $R$-module of $\dim(M) \geq 2$. Let $\aa=a_{1}, \ldots, a_{r}$ be a superficial sequence relative to  $\hdeg_{I}(M)$. If $r < \dim(M)$, then 
\[ \l \left( H^{0}_{\m} \left(M/(\aa)M \right) \right) \leq \hdeg_{I}(M) -  \e_{0}(I, M). \]
\end{Corollary}

%%---------------
\if0
\begin{proof} Let $N=M/(\aa)M$. Then by  Definition \ref{hd1} we obtain
\[ \l \left( H^{0}_{\m} (N) \right) \leq \hdeg_{I}(N) - \e_{0}(I, N) \leq \hdeg_{I}(M) - \e_{0}(I, M). \qedhere \]
\end{proof}
\fi
%%---------------

Brennan and Ciuperc\u{a} showed when the homological degrees are preserved under modulo operation for modules of small dimension. 

\begin{Theorem}{\rm \cite[Theorems 3.1 and 4.2]{BC09}}
Let $R$ be a standard graded ring over a field  and $\m$ the ideal generated by the elements of degree $1$. %Then $R$ is a homomorphic image of a regular ring $S$.  
Let $M$ be a finitely generated graded $R$-module.
\begin{enumerate}[{\rm (1)}]
\item Suppose ${\ds \dim(M)=1}$. Let $a$ be an element of $\m \setminus \m^{2}$ such that ${\ds \dim(M/aM)=0}$. Then ${\ds \hdeg(M) = \hdeg(M/aM)}$ if and only if ${\ds (0 :_{M} a) = H^{0}_{\m}(M) }$.
\item  Suppose ${\ds \dim(M)=2}$. Let $a$ be a homogeneous element of degree $1$ regular on $M/H^{0}_{\m}(M)$. Then ${\ds \hdeg(M) = \hdeg(M/aM)}$ if and only if   ${\ds (0 :_{M} a) = H^{0}_{\m}(M) }$ and ${\ds a \, H^{0}_{\m}( \Ext^{1}_{S}(M, S) ) =0}$. 
\end{enumerate}
\end{Theorem}

Homological torsion is a combinatorial expression of the terms of homological degrees. It behaves well under hyperplane sections \cite[Lemma 2.8]{GO14} and plays a very useful role in the study of Hilbert coefficients.

\begin{Definition}{\rm \cite[Definition 1.2]{GO14} 
Let $(R, \m)$ be a Gorenstein local ring of dimension $d>0$ with infinite residue field $k$.  Let $E=E(k)$ be the injective envelope of $k$. Let $I$ be an $\m$-primary ideal of $R$ and $M$ a finitely generated $R$-module of dimension $s \geq 2$. 
The $i$-th {\em homological torsion} of $M$ with respect to $I$ is
\[ T_{I}^{i}(M) = \sum_{j=1}^{s-i} {{s-i-1}\choose{j-1}} \hdeg_{I} \left( \Hom_{R}( H^{j}_{\m}(M), E )   \right).  \]
In particular, we write $T_{I}(M)= T^{1}_{I}(M)$ and call it the homological torsion. That is,
\[ T_{I}(M) = \sum_{j=1}^{s-1} {{s-2}\choose{j-1}} \hdeg_{I} \left( \Hom_{R}( H^{j}_{\m}(M), E )   \right).  \]
}\end{Definition}

\begin{Proposition}{\rm \cite[Proposition 6.1]{Chern1}}
Let $(R, \m)$ be a Gorenstein local ring.  Let $M$ be a generalized Cohen-Macaulay  $R$-module of dimension $s \geq 2$. Let ${\ds Q=(x_{1}, \ldots, x_{s})}$ be a parameter ideal for $M$.  Then
\[ \l \left( H^{0}_{\m}( M/(x_{1}, \ldots, x_{s-1})M   )   \right) \leq T_{I}(M) \] 
for any $\m$-primary ideal $I$. Moreover, if $M$ is Buchsbaum, then the equality holds.
\end{Proposition}

\begin{proof}  The following is a modification of the proof given in \cite{Chern1} with added details. We prove by induction on $s$. Suppose $s=2$. We may assume that $x=x_1$ is $M$-regular. From the exact sequence ${\ds 0 \rar M \stackrel{x}{\lar} M \rar M/xM \rar 0 }$, we obtain 
\[   H^{0}_{\m}(M) = 0 \lar H^{0}_{\m}(M/xM) \lar H^{1}_{\m}(M) \lar \cdots. \]
Thus, ${\ds \l \left( H^{0}_{\m}(M/xM) \right) \leq \l \left(   H^{1}_{\m}(M) \right) = T_{I}(M)}$. 

\medskip

\noindent Suppose $s \geq 3$ and the assertion is true for $s-1$. Let ${\ds N=M/x_{1}M}$, ${\ds \xx= x_{1}, \ldots, x_{s-1}}$, and ${\ds \xx'=x_{2}, \ldots, x_{s-1}}$. Then  by induction hypothesis, we obtain
\[ \l \left( H^{0}_{\m} (M/ \xx M) \right) = \l \left( H^{0}_{\m} (N/ \xx' N) \right) \leq T_{I}(N). \]
It is enough to show that ${\ds T_{I}(N) \leq T_{I}(M) }$. From the exact sequence ${\ds 0 \rar M \stackrel{x_{1}}{\lar} M \rar N \rar 0 }$, we obtain the long exact sequence
\[ \cdots \lar H^{j}_{\m}(M) \lar H^{j}_{\m}(M) \lar H^{j}_{\m}(N) \lar H^{j+1}_{\m}(M) \lar \cdots.\]
Therefore, for each $j=1, \ldots, s-2$, we have
\[ \l \left( H^{j}_{\m} (N) \right)  \leq  \l \left( H^{j}_{\m}(M) \right) + \l \left( H^{j+1}_{\m}(M) \right). \]
Now we compute ${\ds T_{I}(N)}$.
\[ \begin{array}{rcl}
{\ds T_{I}(N) } &=& {\ds \sum_{j=1}^{s-2} {{s-3}\choose{j-1}} \l \left( H^{j}_{\m}(N) \right) } \vspace{0.1 in} \\ & \leq & {\ds  \sum_{j=1}^{s-2} {{s-3}\choose{j-1}} \l \left( H^{j}_{\m}(M) \right)   +  \sum_{j=1}^{s-2} {{s-3}\choose{j-1}} \l \left( H^{j+1}_{\m}(M) \right)   } \vspace{0.1 in} \\
&=& {\ds \l \left( H^{1}_{\m} (M) \right) + \sum_{j=2}^{s-2} \left({{s-3}\choose{j-1}} + {{s-3}\choose{j-2}}   \right)  \l \left( H^{j}_{\m} (M) \right) + \l \left( H^{s-1}_{\m} (M) \right) }  \vspace{0.1 in} \\
&=& {\ds T_{I}(M).}
\end{array}\]
\end{proof}

\subsection*{Euler Characteristic} Let $Q=(x_{1}, \ldots, x_{s})$ be a parameter ideal for a module $M$. The $i$th Kozul homology module generated by the system of parameters $x_{1}, \ldots, x_{s}$ with coefficients in $M$ is denoted by ${\ds H_{i}(Q, M)}$. The (first) {\em Euler characteristic} of $M$ relative to $Q$ is defined as 
\[ \chi_{1}(Q, M) = \sum_{i \geq 1} (-1)^{i-1} \l \left( H_{i}(Q, M)  \right).  \]
Then by a classical result of Serre \cite{AB58, S65}, we have
\[ \chi_{1}(Q, M) =  \l(M/ QM) -  \e_{0}(Q, M). \]
The Euler characteristic can be also bounded by the homological degree.

\begin{Theorem}\label{Chern5-7-2}{\rm \cite[Theorem 7.2]{Chern5}}\label{HDbound}
Let $(R, \m)$ be a Noetherian complete local ring with infinite residue field and let $M$ be a finitely generated $R$-module with ${\ds \dim(M)=\dim(R)=d \geq 1}$. Then for every parameter ideal $Q$ for $R$, we have
\[ \chi_{1}(Q, M) \leq \hdeg_{Q}(M) - \e_{0}(Q, M).  \]
\end{Theorem}

Goto and Ozeki proved when the equality holds true. Recall that we denote the Hilbert polynomial of $M$ with respect to $Q$ by ${\ds \HP_{Q, M}(X)}$.

\begin{Theorem}\label{GO14-3-3}{\rm \cite[Theorem 3.3]{GO14}}
Let  $(R, \m)$ be a Noetherian complete local ring with infinite residue field and let $M$ be a finitely generated $R$-module with ${\ds \dim(M)=\dim(R)=d \geq 1}$. Let $Q$ be a parameter ideal for $R$. Then
\[ \chi_{1}(Q, M) = \hdeg_{Q}(M) - \e_{0}(Q, M)  \]
if and only if
\[ (-1)^{i} \e_{i}(Q, M) = \left\{ \begin{array}{ll} {\ds T^{\,i}_{Q}(M)} \quad & \mbox{if} \;\; 1 \leq i \leq d-1, \vspace{0.1 in} \\ {\ds \l \left( H^{0}_{\m}(M)  \right) } \quad & \mbox{if} \;\; i=d \end{array} \right.    \]
 for all $1 \leq i \leq d$, and ${\ds \l \left(M/Q^{n+1}M \right) = \HP_{Q, M}(n) }$ for all $n \geq 0$. 
\end{Theorem}

\begin{Example}\label{GO14-3-8}{\rm \cite[Example 3.8]{GO14}
Let $\ell \geq 2$ and $m \geq 1$ be integers. Let 
\[ S=k \bl X_{i}, Y_{i}, Z_{j} \mid 1 \leq i \leq \ell, \; 1 \leq j \leq m \br \quad \mbox{and} \quad  R = S/(X_{1}, \ldots, X_{\ell}) \cap (Y_{1}, \ldots, X_{\ell}), \] where $S$ is the formal power series ring over an infinite field $k$.  Let $x_{i}, y_{i}$, and $z_{j}$ denote the images of $X_{i}, Y_{i}$, and $Z_{j}$ in $R$ respectively.  Consider the following parameter ideal of $R$:
\[ Q = (x_{i} - y_{i} \mid 1 \leq i \leq \ell)  + (z_{j} \mid 1 \leq j \leq m).  \]
Then we have
\[ \chi_{1}(Q, R) = \ell -1, \quad \e_{0}(Q, R) = 2, \quad \mbox{and} \quad \hdeg_{Q}(R) = 2 + {{\ell + m -1}\choose{m+1}}. \]
In particular, if $\ell=2$, then ${\ds \chi_{1}(Q, R) = \hdeg_{Q}(R) - \e_{0}(Q, R)}$. 
}\end{Example}

\subsection*{Chern number} The homological torsion can be used as a bound for the first Hilbert coefficient $\e_{1}(Q, M)$ of a module $M$ relative to an ideal generated by a system of parameters $Q$ for $M$. This coefficient $\e_{1}(Q, M)$ is often called the Chern number, and it has been extensively studied by Vasconcelos and the authors \cite{Chern1, Chern2, Chern3, Chern6, Chern4, Chern5}.

\begin{Theorem}\label{Chern5-6-5}{\rm \cite[Theorem 6.5]{Chern5} }\label{HTbound}
Let  $(R, \m)$ be a Noetherian complete local ring with infinite residue field and let $M$ be a finitely generated $R$-module with ${\ds \dim(R)=\dim(M) \geq 2}$. Let $Q$ be a parameter ideal for $R$. Then
${\ds - \e_{1}(Q, M) \leq T_{Q}(M)}$. 
\end{Theorem}

As a direct consequence of Theorem \ref{GO14-3-3}, we obtain that if ${\ds \chi_{1}(Q, M) = \hdeg_{Q}(M) - \e_{0}(Q, M)}$, then 
${\ds - \e_{1}(Q, M) = T_{Q}(M) }$. The converse is true if $M$ is unmixed. We note that \cite[Example 4.5]{GO14} shows that the unmixedness is necessary.

\begin{Theorem}{\rm \cite[Theorem 4.2]{GO14}}
Suppose that ${\ds \dim(R)= \dim(M) \geq 2}$ and that $M$ is unmixed. Let $Q$ be a parameter ideal for $R$. Then
\[ - \e_{1}(Q, M) = T_{Q}(M) \;\; \mbox{if and only if} \;\;  \chi_{1}(Q, M) = \hdeg_{Q}(M) - \e_{0}(Q, M). \] 
\end{Theorem}

Recall that if $M$ is Cohen-Macaulay, then $\hdeg_{I}(M) = \e_{0}(I, M)$. Thus, sometimes the integer ${\ds \hdeg_{I}(M) - \e_{0}(I, M)}$ is called a {\em Cohen-Macaulay deficiency}. Vasconcelos conjectured that the negativity of $\e_{1}(Q, R)$, where $Q$ is a parameter ideal of $R$, is an expression of the lack of Cohen-Macaulayness of $R$ \cite[Conjecture 3.1]{Chern1}.  Ghezzi, Goto, Hong, Ozeki, Phuong and Vasconcelos settled this conjecture affirmatively for rings in \cite{Chern3} and for modules in \cite{Chern5}. Thus, $\e_{1}(Q, M)$ can be considered also as a Cohen-Macaulay deficiency. The following proposition shows the relation between these two Cohen-Macaulay deficiencies. 
 
\begin{Proposition}\label{Chern1-7-7}{\rm \cite[Corollary 7.7]{Chern1}}
Let $(R, \m)$ be a Noetherian complete local ring with infinite residue field and let $M$ be a finitely generated $R$-module with ${\ds \dim(M)=\dim(R)=d \geq 1}$.  Let $Q=(a_{1}, \ldots, a_{d})$ be a parameter ideal for $R$ such that the sequence ${\ds a_{1}, \ldots, a_{d}}$ is superficial relative to $\hdeg_{Q}(M)$. Then 
\[ -\e_{1}(Q, M) \leq \hdeg_{Q}(M) - \e_{0}(Q, M). \]
\end{Proposition}

\begin{proof} We include a modification of the proof of \cite[Lemma 2.4]{GN03}.  We prove by induction on $d$. If $d=1$, let $Q=(a)$. Let $t$ be a sufficiently large integer such that
\[ H^{0}_{\m}(M) = (0:_{M} Q^{t}), \quad \mbox{and} \quad \l(M/Q^{t}M) = \e_{0}(Q, M) t - \e_{1}(Q, M). \]
Then we obtain the following:
\[ \begin{array}{rcl}
{\ds - \e_{1}(Q, M)} &=& {\ds \l(M/Q^{t}M) - \e_{0}(Q, M) t } \vspace{0.1 in} \\
&=& {\ds \l(M/Q^{t}M) - \e_{0}(Q^{t}, M)} \vspace{0.1 in} \\
&=& {\ds \l(M/Q^{t}M) - \left( \l (M/Q^{t}M) - \chi_{1}(Q^{t}, M) \right) } \vspace{0.1 in} \\
&=& {\ds  \l( H_{1}(Q^{t}, M) ) } \vspace{0.1 in} \\
&=& {\ds \l( (0:_{M} Q^{t})) } \vspace{0.1 in} \\
&=& {\ds \l \left(H^{0}_{\m}(M) \right).}
\end{array} \]
Since $\dim(M)=1$, we have ${\ds \hdeg_{Q}(M) = \e_{0}(Q, M) + \l \left( H^{0}_{\m}(M)  \right)}$. Thus, the assertion follows.

\medskip

\noindent Suppose $d \geq 2$. Let $\aa=a_{1}, \ldots, a_{d-1}$, $S=R/(\aa)$, $q=QS$, and $M'=M/(\aa)M$. Then as proved in the $d=1$ case, we have ${\ds -\e_{1}(q, M') = \l \left( H^{0}_{\m}(M') \right)}$. Therefore,
\[ -\e_{1}(Q, M) \leq -\e_{1}( q, M') = \l \left( H^{0}_{\m}(M') \right) \leq \hdeg_{Q}(M) - \e_{0}(Q, M),\]
where the last inequality follows from Corollary \ref{Chern1-7-2}. 
\end{proof}

By Theorem \ref{Chern5-7-2} and Proposition \ref{Chern1-7-7}, we obtain 
\[ \max \{ \chi_{1}(Q, M), \; - \e_{1}(Q, M) \} \leq \hdeg_{Q}(M) - \e_{0}(Q, M). \]
In \cite{GHV12}, Goto, Hong, and Vasconcelos established the relationships among $\e_{1}(Q)$ and various first Euler characteristics when $Q$ is a parameter ideal of $R$. We state one of the main results of \cite{GHV12}, which is related to the inequality stated above.

\begin{Theorem}{\rm \cite[Theorem 4.2]{GHV12} }
Let $(R, \m)$ be a Noetherian local ring of dimension $d \geq 2$ with ${\ds \depth(R) \geq d-1}$ and infinite residue field. Let $Q=(\xx)$ be a parameter ideal of $R$. Then ${\ds \chi_{1}(Q, M)= - \e_{1}(Q, M)}$ if and only if ${\ds \xx}$ is a $d$-sequence.
\end{Theorem}

The {\em sectional genus} ${\ds g_{s}(I, M)}$ of $M$ with respect to an $\m$-primary ideal $I$ is defined as
\[ g_{s}(I, M) = \l(M/IM) -  \e_{0}(I, M)  + \e_{1}(I, M).  \]
Suppose that ${\ds \dim(R)=\dim(M) }$ and $Q$ is a parameter ideal for $R$. Then due to Serre, we have ${\ds \chi_{1}(Q, M) =  \l(M/ QM) - \e_{0}(Q, M) }$. Thus, we can rewrite
\[ g_{s}(Q, M) =\chi_{1}(Q, M) + \e_{1}(Q, M).\]
As consequences of Theorems \ref{Chern5-7-2} and \ref{Chern5-6-5}, both ${\ds g_{s}(Q, M)}$ and ${\ds \hdeg_{Q}(M) -  \e_{0}(Q, M)  - \,T_{Q}(M)}$ are less than or equal to ${\ds \hdeg_{Q}(M) - \e_{0}(Q, M) + \e_{1}(Q, M)}$. Goto and Ozeki showed a direct relation between ${\ds g_{s}(Q, M)}$ and ${\ds \hdeg_{Q}(M) - \e_{0}(Q, M) - \,T_{Q}(M)}$.

%%-------------
\if0
\[ \begin{array}{rclcl}
{\ds g_{s}(Q, M) } &=& {\ds \textcolor{blue}{\chi_{1}(Q, M)}+ \e_{1}(Q, M)} &\leq& {\ds \textcolor{blue}{\hdeg_{Q}(M) - \e_{0}(Q, M) +  \e_{1}(Q, M)}, \; \mbox{and}  \vspace{0.1 in} \\  && {\ds  \hdeg_{Q}(M) - \e_{0}(Q, M) \;\textcolor{blue}{- \,T_{Q}(M)} } &\leq& {\ds \hdeg_{Q}(M) - \e_{0}(Q, M)  + \textcolor{blue}{\e_{1}(Q, M)} } \end{array}    \]
\fi
%%---------------

\begin{Theorem}{\rm \cite[Proposition 3.3 and Theorem 3.4]{GO16}}
Suppose ${\ds \dim(R)=\dim(M) \geq 2 }$ and let $Q$ be a parameter ideal for $R$. Then
\[  g_{s}(Q, M) \leq \hdeg_{Q}(M) - \e_{0}(Q, M)  - T_{Q}(M). \]
The equality holds true if and only if
\[ (-1)^{i} \e_{i}(Q, M) = \left\{ \begin{array}{ll}  
{\ds T^{i}_{Q}(M)} \quad & \mbox{if} \;\; 2 \leq i \leq d-1, \vspace{0.1 in} \\
{\ds \l \left( H^{0}_{\m}(M) \right) } & \mbox{if} \;\;  i=d,
\end{array}  \right. \]
for all $2 \leq i \leq d$ and 
\[ \l \left( M/QM  \right) = \sum_{i=0}^{d}(-1)^{i}\e_{i}(Q, M).   \]
\end{Theorem}

Let $(R, \m)$ be a Cohen-Macaulay local ring, $I$ an $\m$-primary ideal, and $Q$ a minimal reduction of $I$. 
Rossi proved that, if ${\ds \dim(R) \leq 2}$, then ${\ds g_{s}(I, R) \geq r_{Q}(I) -1}$, where $r_{Q}(I)$ is the reduction number of $I$ with respect to $Q$ \cite[Corollary 1.5]{R00}. Thus, the homological degree can be used as an upper bound for a reduction number. Ghezzi, Goto, Hong, and Vasconcelos were able to extend Rossi's result to $2$-dimensional Buchsbaum rings \cite[Theorem 4.3]{red}. We conclude this section with several open questions raised by Vasconcelos.

\begin{Question}{\rm
Let $R$ be a Cohen-Macaulay local ring which is a homomorphic image of a Gorenstein local ring. Let $\mathcal{M}(R)$ be the category of all finitely generated $R$-modules. Consider the following set of rational numbers:
\[ \left\{ \frac{\hdeg(M) - \hdeg(M/hM)}{ \deg(M) }   \;\;\vline\;\; M \in \mathcal{M}(R), \; \mbox{and  $h$ is a generic hyperplane section}  \right\}.  \]
Is this set finite or bounded? Can it be expressed as an invariant of $R$?
}\end{Question}

\begin{Question}{\rm
Let $(R, \m)$ be a homomorphic image of a Gorenstein local ring and let $Q=(\xx)$ be a parameter ideal for $R$. If $\xx$ is a $d$-sequence, what is an estimation for $\hdeg_{Q}(R)$? 
}\end{Question}

%\begin{Question}{\rm
%Let $R=k[\Delta]$ be the Stanley-Reisner ring of the simplicial complex $\Delta$. Find estimations for ${\ds \hdeg(R)}$ and ${\ds \bdeg(R)}$.
%}\end{Question}

\begin{Question}{\rm
Let ${\ds \varphi: (R, \m) \rar (S, \n)}$ be a flat homomorphisms of local rings which are homomorphic images of a Gorenstein local rings. Can we relate ${\ds \hdeg(R), \hdeg(S)}$, and ${\ds \hdeg(S/\m S)}$?
}\end{Question}

\medskip

\section{Cohomological Degree}\label{CDsection}

%\section{Cohomological Degree and $\bdeg$}\label{CDsection}

Cohomological degrees were introduced by Vasconcelos and generalize the homological degree.

\begin{Definition}\label{Deg}{\rm (\cite[Definition 3.1]{V98}, \cite{V98-2})
Let $R$ be either a Noetherian local ring or a standard graded algebra with the (irrelevant) maximal ideal $\m$. Let $\mathcal{M}(R)$ be the category of finitely generated $R$-modules. The {\em cohomological degree} (or {\em extended degree}) on $\mathcal{M}(R)$ is a function ${\ds \Deg: \mathcal{M}(R) \rar \RR }$ satisfying the following three conditions:
\begin{enumerate}[(i)]
\item ${\ds \Deg(M) = \Deg \left(M/H^{0}_{\m}(M) \right) + \l \left(  H^{0}_{\m}(M) \right)}$.
\item Suppose that $R$ has positive depth and that $h \in R$ is a regular, generic hyperplane section on $M$. Then ${\ds \Deg(M/hM) \leq \Deg(M)}$.
\item If $M$ is a Cohen-Macaulay module, then ${\ds \Deg(M) =  \deg(M)}$.
\end{enumerate}
}\end{Definition}

If $(R, \m)$ is a homomorphic image of a Gorenstein local ring and $I$ an $\m$-primary ideal, then the homological degree $\hdeg_{I}$ with respect to $I$ is a cohomological degree of $R$ (Proposition \ref{V98-3-18}, Corollary \ref{Chern1-7-2}).

%Moreover, by \cite[Proposition 7.3]{V98}, ${\ds \hdeg(M) \geq \hdeg(M_{\p}) }$ for any prime ideal $\p$. 
%\begin{Question}{\rm
%What is a relationship between ${\ds \Deg(M) }$ and ${\ds \Deg(M_{\p})}$?
%}\end{Question}

\begin{Remark}{\rm
There is just one cohomological degree in dimension at most 1.
\begin{enumerate}[(1)]
\item If ${\ds \dim(M)=0}$, then ${\ds \Deg(M) = \l(M)}$.
\item If ${\ds \dim(M)=1}$, then ${\ds \Deg(M) =  \deg(M) + \l \left( H^{0}_{\m}(M)  \right) = \adeg(M) = \hdeg(M)}$. 
\end{enumerate}
}\end{Remark}

An important property of the cohomological degree is that it provides an estimate for the Castelnuovo-Mumford regularity $\reg(A)$ of a standard graded algebra $A$.

\begin{Theorem}{\rm \cite[Theorem 2.4]{V98-2}}\label{CDmumford}
Let $A$ be a standard graded algebra over an infinite field $k$. Then
$\reg(A) < \Deg(A)$.
\end{Theorem}

%We denote the minimal number of generators of $M$ by $\nu(M)$. 
Another important property of the cohomological degree is that it bounds the minimal number of generators of a module.

\begin{Proposition}\label{V98-2-2-1}{\rm \cite[Proposition 2.1]{V98-2}}
Let  $M$ be a finitely generated $R$-module.
\begin{enumerate}[{\rm (1)}]
\item ${\ds \Deg(M) \geq \deg(M)}$ and the equality holds if and only if $M$ is Cohen-Macaulay.
\item  ${\ds \nu(M) \leq \Deg(M)}$.
\item Let $L$ be a submodule of $M$. If ${\ds \l(L) < \infty}$, then ${\ds \Deg(M) = \Deg(M/L) + \l(L) }$. 
\end{enumerate}
\end{Proposition}

\begin{Proposition}{\rm \cite[Proposition 3.9]{V98}}
Let $(R, \m)$ be a Noetherian local ring and $M$ a finitely generated $R$-module. Then
\[ \deg(M) \leq \adeg(M) \leq \Deg(M). \]
\end{Proposition}

\begin{proof} We collect the associated primes of $M$ by their dimensions: ${\ds \dim(M) =s_{1} > s_{2} > \ldots > s_{n} }$. Then
\[ \adeg(M) = a_{s_{1}}(M) + a_{s_{2}}(M) + \cdots + a_{s_{n}}(M), \]
where $a_{s_{i}}(M)$ is the contribution of all primes in $\Ass(M)$ of dimension $s_{i}$. Also, ${\ds a_{s_{1}}(M) = \deg(M)}$, which proves the first inequality. Since both $\adeg$ and $\Deg$ are additive on the exact sequence ${\ds 0 \rar H^{0}_{\m}(M) \rar M \rar M/H^{0}_{\m}(M) \rar 0}$, we may assume that $H^{0}_{\m}(M)=0$. Let $h$ be a regular hyperplane section on $M$. Then by Theorem \ref{V98-1-14}, the induction hypothesis  on $\dim(M)$, and Definition \ref{Deg}, we obtain
\[ \adeg(M) \leq \adeg(M/hM) \leq \Deg(M/hM) \leq \Deg(M). \qedhere \]
\end{proof}

Cohomological degrees can be used to estimate the minimal number of generators of an ideal in an arbitrary local ring. The following theorem generalizes a bound that holds when $R$ is Cohen-Macaulay \cite[Theorem 4.1]{V98-2}, and it improves previously known bounds (see \cite[Section 3]{V98-2}).

\begin{Theorem}\label{V98-2-4-6}{\rm \cite[Theorem 4.6]{V98-2}}\label{CDbound}
Let $(R, \m)$ be a local ring of dimension $d \geq 1$ with infinite residue field. Let $I$ be an $\m$-primary ideal. Suppose that ${\ds \m^{s} \subseteq I}$. Then, for any cohomological degree $\Deg$ on $\mathcal{M}(R)$, we have
\[ \nu(I) \leq \Deg(R) \, {{s+d-2}\choose{d-1}} + {{s+d-2}\choose{d-2}}. \] 
\end{Theorem}

Gunston introduced the following cohomological degree in his Ph.D. thesis. 

\begin{Definition}{\rm \cite[Definition 3.1.1]{G98} Let $\mathcal{D}(R)$ denote the set of all cohomological degree functions on $\mathcal{M}(R)$. For a finitely generated $R$-module, we define
\[ \bdeg(M) = \min \{ \Deg(M) \mid \Deg \in \mathcal{D}(R) \}. \]
}\end{Definition}

The inductive approach for $\bdeg$ works.

\begin{Theorem}\label{G98-3-1-2}{\rm \cite[Theorem 3.1.2]{G98}}
Let $M$ be a finitely generated $R$-module of positive depth. Then for a generic hyperplane $h$, we have
\[ \bdeg(M) = \bdeg(M/hM). \]
\end{Theorem}

\begin{Corollary}
Let $M$ be a finitely generated $R$-module of dimension $s$. There exist a generic superficial sequence ${\ds x_{1}, \ldots, x_{s}}$ for $M$ such that
\[ \begin{array}{rcl}
{\ds \bdeg(M) } & \leq & {\ds \l \left( H^{0}_{\m}(M)  \right) + \sum_{i=1}^{s} \l \left( H^{0}_{\m} (M/(x_{1}, \ldots, x_{i})M) \right) } \vspace{0.1 in} \\
&\leq & {\ds (s+1) \hdeg(M) - s \cdot\deg(M)  }
\end{array}\]
\end{Corollary}

\begin{proof} The first inequality is a consequence of Theorem \ref{G98-3-1-2} and the second follows from Corollary \ref{Chern1-7-2}.
\end{proof}

One of the difficulties of computing cohomological degrees lies on their behavior on short exact sequences. It turns out $\bdeg$ has more amenable properties.

\begin{Proposition}{\rm (\cite[Proposition 3.2.2]{G98},  \cite[Propositions 3.2 and 3.3]{D07}) }

\noindent Let ${\ds 0 \rar L \rar M \rar N \rar 0}$ be an exact sequence of finitely generated $R$-modules.
\begin{enumerate}[{\rm (1)}]
\item If ${\ds \l(N) < \infty}$, then ${\ds \bdeg(M) \leq \bdeg(L) + \l(N) }$.
\item ${\ds \bdeg(M) \leq \bdeg(L) + \bdeg(N) }$.
\item ${\ds \bdeg(L) \leq \bdeg(M) + \left( \dim(L) - 1 \right) \bdeg(N)  }$.
\end{enumerate}
\end{Proposition}

We denote ${\ds M^{*}= \Hom_{R}(M, R) }$ for an $R$-module $M$.

\begin{Theorem}{\rm \cite[Theorems 5.2 and 5.3]{D07}}
Let $(R, \m)$ be a Gorenstein local ring of dimension $d$ and $M$ a finitely generated $R$-module. Then
\[ \nu_{R}(M^{*}) \leq \bdeg_{R}(M^{*}) \leq \hdeg(M) \left( \deg(R) + \frac{d(d-1)}{2}  \right). \]
\end{Theorem} 

%%-------------------------
% Omit
%\begin{Theorem}
%Let $(R, \m)$ be a Buchsbaum local ring with infinite residue field. For any parameter ideal $Q$, we have
%\[ \bdeg_{Q}(R) = \l( R/Q). \]
%\end{Theorem}
%
%\begin{proof} \colorbox{yellow}{Comment.} This is in Wolmer's degree note (Theorem 5.60). Not sure what he meant by ${\ds \bdeg_{Q}(R)}$. Needs to check it out.
%\end{proof}
%%---------------------------

\section{The canonical degree}\label{canonicalsection}

Let $(R, \m)$ be a Cohen-Macaulay local ring that has a canonical ideal $\C$. The canonical degree was introduced in \cite{blue1} with the overall goal to refine our understanding of $R$ and to measure the deviation from $R$ being Gorenstein.
%First we recall the definition and the important properties. 
The canonical degree requires knowledge of the multiplicity of $\m$-primary ideals.

\medskip

We begin with the definition of canonical degree in $1$-dimensional Cohen-Macaulay local rings, which is our main setting. We denote length by $\lambda$, and the multiplicity associated with the $\m$-adic filtration by $\deg(\tratto)$.

 \begin{Definition}\label{cdeg}
 {\rm Let $R$ be a $1$-dimensional Cohen-Macaulay local ring with a canonical ideal $\C$. The {\em canonical degree of $R$} is the integer
 $\cdeg(R)=\e_0(\C) -\l(R/\C)$. In particular, if $(a)$ is a minimal reduction of $\C$, then  $\cdeg(R)=\l(\C/(a))\geq 0$.}
 %Then the integer $\cdeg(\RR)=\e_0(\C) -\l(\RR/\C)$ is independent of the canonical ideal $\C$.
 \end{Definition}
 
We show in \cite[Proposition 2.1]{blue1} that $\cdeg(R)$ is well defined; that is, the integer $\e_0(\C) -\l(R/\C)$ is independent of the canonical ideal $\C$.

\begin{Remark}{\rm In $1$-dimensional Cohen-Macaulay local rings the vanishing of $\cdeg(R)$ is equivalent to $R$ being Gorenstein, since we have $\cdeg(R)=0$ if and only if $\C$ is a principal ideal.}
\end{Remark}

%Add proof?

%\begin{proof}  If $x$ is an indeterminate over $\RR$, in calculating these differences we may pass from $\RR$
%to $\RR(x) = \RR[x]_{\m \RR[x]}$, in particular we may assume that the ring has an infinite residue field.

%Let $\C$ and $\mathcal{D}$ be two canonical ideals. Suppose $(a)$ is a minimal reduction of $\C$.
%Since $\mathcal{D} \simeq \C$ (\cite[Theorem 3.3.4]{BH}), $\mathcal{D} = q \C$ for some fraction $q$.  If $\C^{n+1} = (a) \C^n$
% by multiplying it by $q^{n+1}$,  we get $\mathcal{D}^{n+1} =(qa) \mathcal{D}^n$, where $(qa) \subset \mathcal{D}$. Thus $(qa)$ is a reduction of $\mathcal{D}$ and
%$\C/(a) \simeq \mathcal{D}/(qa)$. Taking their co-lengths we have
%\[ \l(\RR/(a)) - \l(\RR/ \C) = \l(\RR/(qa)) - \l(\RR/ \mathcal{D}).\]
  %Since $\l(\RR/(a)) = \e_0(\C)$ and $\l(\RR/(qa)) = \e_0(\mathcal{D})
% $,
%we have
%\[ \e_0(\C) - \l(\RR/ \C) = \e_0( \mathcal{D}) - \l(\RR/ \mathcal{D}).\]
  %\end{proof}
  
 We define $\cdeg(R)$ in full generality as follows.
 
 \begin{Definition}\label{defcdeg}
% {\rm \cite[Theorem 2.2, Definition 2.3]{blue1}
{\rm Let $R$ be a Cohen-Macaulay local ring  of dimension $d \geq 1$ that has a canonical ideal $\C$. The {\em canonical degree of $R$} is the integer
  \[ \cdeg(R)= \sum_{\tiny \h(\p)=1} \cdeg(R_{\p}) \deg(R/\p) = \sum_{\tiny \h(\p)=1} [\e_{0}(\C_{\p}) - \l((R/\C)_{\p})] \deg(R/\p).\]
}\end{Definition}

We show in \cite[Theorem 2.2]{blue1} that $\cdeg(R)$  is a well-defined finite sum independent of the canonical ideal $\C$. The definition yields a nice formula in terms of multiplicity when $\C$ is equimultiple, and several results in dimension one extend to higher dimension when $\C$ is equimultiple.

\begin{Remark}{\rm If the canonical ideal $\C$ is equimultiple with a minimal reduction $(a)$, then by the associativity formula \cite[Corollary 4.7.8]{BH} we have
\[  \cdeg(R) = \deg(\C/(a)) = \e_0(\m, \C/(a)). \]
}
\end{Remark}

In general, the vanishing of 
$\cdeg(R)$ characterizes Gorensteinness in codimension one.

\begin{Remark}{\rm 
%Let $(\RR, \m)$ be a Cohen-Macaulay local ring  of dimension $d \geq 1$ that has a canonical ideal $\C$.
$\cdeg(R)\geq 0$ and vanishes if and only if $R$ is Gorenstein in codimension $1$.}
% Furthermore, suppose that $\C$ is equimultiple. Then $\cdeg(\RR) =0$ if and only if $\RR$ is Gorenstein.
\end{Remark}

\begin{Remark}{\rm If $R$ is Gorenstein, it is clear that $\cdeg(R)=0$. However, the converse does not hold in general. For example, if $R$ is a normal domain we have that $\cdeg(R)=0$, since $R_{\p}$ is Gorenstein for every prime $\p$ of $R$ with $ \h(\p)\leq 1$. Therefore if $R$ is a non-Gorenstein normal domain then $\cdeg(R)=0$}.
\end{Remark}

Generalizing the one-dimensional case, we observe that when $\C$ is equimultiple the vanishing of $\cdeg(R)$ is equivalent to $R$ being Gorenstein. We denote the type of $R$ by $r(R)$.

\begin{Corollary}\label{cdegr1}\cite[Corollary 2.5]{blue1}
%Let $(\RR, \m)$ be a Cohen-Macaulay local ring   that   has a canonical ideal $\C$.
Suppose that the canonical ideal of $R$ is equimultiple. Then we have the following.
\begin{enumerate}[{\rm (1)}]
\item $\cdeg(R) \geq r(R)-1$.
\item $\cdeg(R) =0$ if and only if $R$ is Gorenstein.
\end{enumerate}
\end{Corollary}

\begin{proof}
Let $(a)$ be a minimal reduction of the canonical ideal $\C$. Then
\[ \cdeg(R) = \e_0(\m, \C/(a))\geq \nu(\C/(a))=r(R)-1.\]
If $\cdeg(R)=0$ then $r(R)=1$, which proves that $R$ is Gorenstein.
\end{proof}

In light of the above corollary, we are interested in examining extremal values of the canonical degree. We see that the minimal value for the canonical degree is assumed on a new class of Cohen-Macaulay rings, called almost Gorenstein rings, introduced
in \cite{BF97}. A rich theory has been developed in dimension one by Goto, Matsuoka and Phuong \cite{GMP11}, and in higher dimension by Goto, Takahashi and Taniguchi \cite{GTT15}. We recall the definition.
\medskip

\begin{Definition} \cite[Definition 3.3]{GTT15} {\rm A Cohen-Macaulay local ring $R$ with a canonical module $\omega$  is said to be an {\em almost Gorenstein} local ring (AGL for short) if there exists an exact sequence of $R$-modules
${\ds 0 \rightarrow R \rightarrow \omega\rightarrow X \rightarrow  0}$
such that $\nu(X)=\e_0(X)$, where $\nu(X)$ is the minimal number of generators of $X$. In particular, if $R$ is Gorenstein, then $X=0$, so $R$ is almost Gorenstein. }
\end{Definition}

%Note that if $\dim(R)=1$, then $X \simeq (R/\m)^{r(R)-1}$
The following result characterizes AGL rings in terms of minimal canonical degree. We include the proof for completeness.

\begin{Proposition}\label{1dimalmostg}{\rm \cite[Proposition 3.3]{blue1}}
Let $(R,\m)$ be a $1$-dimensional Cohen-Macaulay local ring with a canonical ideal $\C$.
 Then $\cdeg(R)= r(R)-1$ if and only if $R$ is an almost Gorenstein ring.
\end{Proposition}
\begin{proof} We may assume that $R$ is not a Gorenstein ring.  Suppose that $\cdeg(R)= r(R)-1$. Let $(a)$ be a minimal reduction of $\C$. Consider the exact sequence of $R$-modules
\[ 0 \to R \overset{\varphi}{\to} \C \to X \to 0, \;\; \mbox{where} \;\; \varphi (1) = a. \] Then $\nu(X)=r(R) -1=\cdeg (R)=\e_0(X)$. Thus, $R$ is an almost Gorenstein ring.

\medskip

\noindent For the converse, assume that $R$ is almost Gorenstein. There exists an exact sequence of $R$-modules
\[ 0 \to R \overset{\varphi}{\to} \C \to X \to 0 \;\; \mbox{such that} \;\; \nu(X)=\e_0(X).\]  Since $\dim(X)=0$ by
\cite[Lemma 3.1]{GTT15}, we have $\lambda(X)=\e_0(X)=\nu(X)=\lambda(X/\m X)$, so that $\m X=(0)$.
Let $a = \varphi (1) \in \C$ and set $Q = (a)$. Then $\m Q\subseteq \m \C \subseteq Q$. Therefore, since $R$ is not a discrete valuation ring and $\l(Q/\m Q) = 1$, we get $\m \C = \m  Q$, so that $Q$ is a minimal reduction of $\C$.
 %by the Cayley-Hamilton theorem, 
It follows that $\cdeg (R) = \e_0(X) = \nu(X) =r(R) - 1$.
\end{proof}

\begin{Remark}\label{cdeg1}{\rm Assume that $\C$ is equimultiple. The proof of Proposition~\ref{1dimalmostg} shows that if $\cdeg(R)= r(R)-1$, then $R$ is an almost Gorenstein ring \cite[Proposition 3.3]{blue1}. 
%It follows that if $\cdeg(R) \le 1$, then $R$ is an almost Gorenstein ring. 
In particular, if $\cdeg(R)=1$, then by Corollary \ref{cdegr1}, $R$ is not Gorenstein and $r(R)=2$. Therefore $R$ is almost Gorenstein (not Gorenstein).}
\end{Remark}

Under a more technical assumption we can generalize Proposition~\ref{1dimalmostg} to higher dimension. We refer the reader to \cite[Theorem 3.5]{blue1}.

 \subsection*{Numerical semigroup rings}\label{subsectionnumerical} Several interesting results and examples can be found in numerical semigroup rings. We briefly recall the definitions and basic properties that we will use throughout the manuscript. Our focus will be on  
  3-generated numerical semigroup rings. For reference we use \cite[Section 8.7]{VilaBook}.

  \medskip

Given positive integers $a_{1}, \ldots, a_{n}$, let  $H = <a_1, \ldots, a_n> = \{c_{1}a_{1} + \cdots + c_{n} a_{n} \mid c_{i} \in \ZZ_{\geq 0}  \}$ be the numerical semigroup generated by $a_{1}, \ldots, a_{n}$. We assume that $\gcd(a_1, \ldots, a_n) = 1$.  Let $k$ be a field and let $k\bl t \br$ be the formal power series ring over $k$. We denote the subring of $k\bl t \br$
  generated by the monomials $t^{a_i}$ by $k\bl H \br$ and we call it the semigroup ring of $H$ over $k$. The ring $k\bl H \br$ is a one-dimensional Cohen-Macaulay local domain. 

\medskip

There are several integers playing roles in deriving properties of $k \bl H \br$. 
There is an integer $s$ such that $t^n\in k \bl H \br$ for all $n\geq s$. The smallest such $s$ is called the {\em conductor} of $H$ or of $k\bl H \br$, which we denote by $c$, and $t^c k \bl t \br$ is the largest ideal of $k \bl t \br$ contained in $k \bl H \br$. Then  $c-1$ is called the Frobenius number of $k \bl H \br$ and expresses its multiplicity, $c-1=\deg(k \bl H \br)$. For the computation of the canonical module of $k \bl H \br$ see for instance \cite[Section 6]{CGKM}.

%\medskip

%For any positive integer $a$ of $H$, the subring $\AA=k[t^a]$ is a Noether normalization of $k[H]$. This permits the passage of many
%properties from $\AA$ to $k[H]$, and vice-versa. Then $k[H]$ is a free $\AA$-module and taking into account the natural graded structure
%we can write
%\[ k[H] = \bigoplus_{j=1}^m \AA t^{\alpha_j}.\]
%Note that $s \leq \sum_{j=1}^m  \alpha_j$. 
%An interesting question is to determine other invariants of $k[H]$ such as its canonical ideal $\C$ and the reduction number of $\C$ and the canonical degrees
%$\cdeg(k[H])$ and $\ddeg(k[H])$.

\subsubsection*{3-generated numerical semigroup rings}\label{sectionabc}
Let $k$ be a field, let $a, b, c$ be positive integers with $\gcd(a,b,c)=1$, $H=<a,b,c>$ and $R = k\bl t^a, t^b, t^c \br$.
%with $a<b< c$ and 
Assume $R$ is not Gorenstein (i.e., $H$ is not symmetric). Then $r(R)=2$ (see \cite[Section 4]{GMP11}). 
There exists a surjective homomorphism $\Gamma: S= k \bl X, Y, Z \br \rar R$ given by $\Gamma(X)=t^{a}, \Gamma(Y)=t^{b}, \Gamma(Z)=t^{c}$. 
Let $P= \ker(\Gamma)$. Then there exists a matrix
 $$\left(\begin{matrix}
X^{a_1} & Y^{b_1} & Z^{c_1}\\
Y^{b_2}&Z^{c_2}&X^{a_2}\\
\end{matrix}\right)$$
%Then there exists a matrix \[{\ds \varphi = \left[ \begin{array}{lll}
%X^{a_1} & Y^{b_1} & Z^{c_1}\\
%Y^{b_2} & Z^{c_2} & X^{a_2}\\
%\end{array}
%\right]} \] 
such that $P=I_{2}(\varphi)$, the ideal generated by $2 \times 2$ minors of $\varphi$. That is, 
\[ P = ( f_{1}, f_{2}, f_{3} ) = (X^{a_{1}+a_{2}} - Y^{b_{2}}Z^{c_{1}}, \; Y^{b_{1}+b_{2}} - X^{a_{1}}Z^{c_{2}}, \; Z^{c_{1}+c_{2}} - X^{a_{2}}Y^{b_{1}} ).  \]
We call this matrix the Herzog matrix of the semigroup $H$ \cite{H1970}. The computation of the Herzog matrix is also summarized in \cite[Section 3]{HHS21}.

\medskip

We recall a way to compute the canonical ideal in this case. See also \cite[Theorem 6.25]{V2}.

\begin{Remark}{\rm Let $x, y, z$ be the images of $X, Y, Z$ in $R \simeq S/P$ respectively.
Let $L=(f_{1}, f_{2})$.  Since $f_{1}, f_{2}$ form a regular sequence, the canonical module of $R$ is
\[
\C =   \Ext^{2}_{S} (R, S)  \simeq   \Ext^{2}_{S} ( S/P, S) \simeq \Hom_{S/L}( S/P, S/L) \simeq (L: P)/ L = (L: f_{3})/L  \simeq (x^{a_{1}}, y^{b_{2}}). \] 
Similarly,  $\C \simeq (y^{b_{1}}, z^{c_{2}})\simeq (x^{a_{2}}, z^{c_{1}})$.}
\end{Remark}

%$\bullet$ We should get other isomorphisms using $f_{1}, f_{3}$ and $f_{2}, f_{3}$. Is every pair a regular sequence? I think so because we mean regular sequence in the polynomial ring and they have no common factors.

 Goto, Matsuoka and Phuong proved a nice formula for the canonical degree \cite[Theorem 4.1]{GMP11}. 
 
 \begin{Theorem}\label{abc} With notation as above we have
 \begin{enumerate}[{\rm (1)}]
\item  $\cdeg(R)=a_1b_1c_1$, if $bb_2-aa_1>0${\rm ;}
\item  $\cdeg(R)=a_2b_2c_2$, if $bb_2-aa_1<0$.
 \end{enumerate}
 \end{Theorem}
 
 \begin{Corollary}\label{AGmatrix} \cite[Corollary 4.2]{GMP11} The ring $k \bl t^a,t^b,t^c \br$ is an almost Gorenstein ring if and only if the Herzog matrix is either 
${\ds  \left(\begin{matrix}
X&Y &Z\\
Y^{b_2}&Z^{c_2}&X^{a_2}\\
\end{matrix}\right)}$ or ${\ds \left(\begin{matrix}
X^{a_1}&Y^{b_1} &Z^{c_1}\\
Y&Z&X\
\end{matrix}\right)}$. 

 \end{Corollary}
 
 \begin{Example} {\rm Let $R=k \bl t^3,t^4,t^5 \br $. The Herzog matrix is ${\ds \left(\begin{matrix}
X&Y &Z\\
Y&Z&X^2\
\end{matrix}\right)}$, so $\cdeg(R)=1$. $R$ is almost Gorenstein. }
 \end{Example}
 
  \begin{Example}\label{579} {\rm Let $R=k \bl t^5,t^7,t^9 \br$. The Herzog matrix is $\left(\begin{matrix}
X&Y &Z^2\\
Y&Z&X^4\
\end{matrix}\right)$, so $\cdeg(R)=2$. $R$ is not almost Gorenstein. }
 \end{Example}

 \section{Canonical index and Sally module}\label{sallysubsection}
In view of Corollary \ref{cdegr1} and Proposition \ref{1dimalmostg}, the next case to consider is $\cdeg(R)= r(R)$. It is open in general. Some examples are provided in \cite[Example 4.12]{blue1}. In the smallest case, $\cdeg(R)= r(R)=2$, we obtained nice relationships among several invariants of $R$ in \cite[Section 4]{blue1}. Inspired by \cite{blue1} and \cite{sally}, Chau, Goto, Kumashiro and Matsuoka defined the notion of 2-AGL rings as a generalization of AGL rings of dimension one \cite{CGKM}. The theory of these rings has then been developed by Goto, Isobe, and Taniguchi \cite{GIT}.  Both \cite{blue1}  and  \cite{CGKM} use the Sally module, introduced by Vasconcelos in \cite{SallyMOD}. In order to discuss these results, we recall the necessary background. 
 
 \medskip
 
 We showed in \cite[Proposition 4.1]{blue1} that the reduction number of the canonical ideal of $R$ is an invariant of the ring. Therefore we have the following definition.

  \begin{Definition}{\rm \cite[Definition 4.2]{blue1}
Let $R$ be a Cohen-Macaulay local ring of dimension $d \geq 1$ with a canonical ideal $\C$. The {\em canonical index} of $R$ is the reduction number of the canonical ideal $\C$ of $R$ and is denoted by $\rho(R)$.
}\end{Definition}

\begin{Remark}{\rm Assume $d=1$ and that $R$ is not Gorenstein. Let $(a)$ be a minimal reduction of $\C$. Notice that we can not have $\rho(R)=1$, otherwise $\C^2 = a\C$. Then  $\C a^{-1} \subset \Hom(\C , \C) = R$ so that $\C = (a)$, a contradiction.
Furthermore, if $R$ is almost Gorenstein (i.e., $\cdeg(R)=r(R) -1$), then $\rho(R)=2$, by \cite[Theorem 3.16]{GMP11}. We note that the converse is not true, as shown in the following example. }
\end{Remark}
 
\begin{Example}\label{squares}{\rm \cite[Example 4.9]{blue1}  Let $S=k \bl X,Y,Z \br$, let $I=(X^4-Y^{2}Z^{2}, Y^{4}-X^{2}Z^{2}, Z^{4}-X^{2}Y^{2})$ and $R=S/I$. Let $x, y, z$ be the images of $X, Y, Z$ in $R$. Then $\C=(x^{2}, z^{2})$ is a canonical ideal of $R$ with a minimal reduction $(x^{2})$. We have that $\rho(R)=2$, $\e_1(\C)=16$ and $\cdeg(R)=8$.}
\end{Example}

 %The next case to consider is $\cdeg(R)= r(R)=2$. Assume that $\dim(R)=1$. Maybe give background on the canonical index and \cite[Theorem 4.6]{blue1} for a nice relation among the invariants.

 %\subsubsection*{Sally module}  
 We briefly recall the Sally module associated to
the canonical ideal $\C$ in rings of dimension $1$. Let $Q=(a)$
be a minimal reduction of  $\C$ and consider the exact sequence of finitely generated $R[QT]$-modules
\[ 0 \rar \C R[QT] \rar \C R[\C T] \rar S_Q(\C) \rar 0. \]
Then the Sally module  $S=S_Q(\C) = \bigoplus_{n\geq 1} \C^{n+1}/\C Q^{n}$ of $\C$ relative to $Q$ is Cohen-Macaulay and, by
 \cite[Theorem 2.1]{red}, we have
 \begin{equation}\label{eq1}
 \e_1(\C) = \cdeg(R) + \sum_{j=1}^{\rho(R)-1} \lambda(\C^{j+1}/a\C^j) = \sum_{j=0}^{\rho(R) -1} \lambda(\C^{j+1}/a\C^j). 
 \end{equation}
%(i), (ii): Since $\depth \gr_{\C}(\RR) \geq d-1$, $S$ is Cohen-Macaulay.

The following property of Cohen-Macaulay rings of type $2$ is a useful calculation that we use to characterize rings with minimal canonical index $2$.

\begin{Proposition}\label{C2}\cite[Proposition 4.5]{blue1}
Let $R$ be a $1$-dimensional Cohen-Macaulay local ring with a canonical ideal $\C$.  Let $(a)$ be a minimal reduction of $\C$. If $\nu(\C)=2$, then $\l(\C^{2}/a \C)=\l(\C/(a))$.
\end{Proposition}

From Equation (\ref{eq1}) and Proposition \ref{C2} we obtain the following (see also Example \ref{squares}).

\begin{Theorem}\label{SallyofC} \cite[Theorem 4.6]{blue1}  Let $R$ be a $1$-dimensional Cohen-Macaulay local ring with a canonical ideal $\C$. Suppose that the type of $R$ is $2$.
Then $\rho(R)=2$ if and only if ${\ds \e_{1}(\C) = 2 \, \cdeg(R)}$.
\end{Theorem}

%We now characterize rings with minimal canonical index $\rho(R)=2$. We include the proof for completeness.

%\begin{Theorem}\label{SallyofC} \cite[Theorem 4.6]{blue1}. Let $ (R, \m)$ be a $1$-dimensional Cohen-Macaulay local ring with a canonical ideal $\C$. Suppose that the type of $R$ is $2$.
%Then we have the following. (Maybe include only part 2).
%\begin{itemize}
%\item[{\rm (1)}] ${\ds \e_{1}(\C) \leq \rho(R) \, \cdeg(R)}$.
%\item[{\rm (2)}] $\rho(R)=2$ if and only if ${\ds \e_{1}(\C) = 2 \, \cdeg(R)}$.
%\end{itemize}
%\end{Theorem}

%\begin{proof} Let $\C = (a, b)$, where $(a)$ is a minimal reduction of $\C$.

%\medskip

%\noindent (1) For each $j=0, \ldots, \rho(R) -1$,  the module $\C^{j+1}/a\C^{j}$ is cyclic and annihilated by $L=\ann(\C/(a))$ . Hence we obtain
%\[\e_1(\C) = \sum_{j=0}^{\rho(R) -1} \lambda(\C^{j+1}/a\C^j)  \leq  \rho(R) \, \lambda(R/L)= \rho(R)  \, \cdeg(R).\]
%(2) Note that $\rho(R)=2$ if and only if ${\ds \e_{1}(\C) = \sum_{j=0}^{1} \lambda(\C^{j+1}/a\C^j) }$.
%Since $\nu(\C)= r(R)=2$, by Proposition~\ref{C2}, $\l(\C/(a) ) = \l(\C^{2}/a \C)$.  Thus, the assertion follows from
%\[  \sum_{j=0}^{1} \lambda(\C^{j+1}/a\C^j) = 2 \, \l(C/(a)) = 2 \, \cdeg(R).\]
%\end{proof}

%\medskip

\begin{Remark} \label{Sallymult}{\rm Let $R$ be a $1$-dimensional Cohen-Macaulay local ring with a canonical ideal $\C$. Then the multiplicity of the Sally module $s_0(S)=\e_1(\C)-\e_0(\C)+\lambda(R/\C)=\e_1(\C)-\cdeg(R)$ is an invariant of $R$, by \cite[Corollary 2.8]{GMP11} and \cite[Proposition 2.1]{blue1}.} \end{Remark}

\begin{Remark}\label{relationships}{\rm Let $R$ be a $1$-dimensional Cohen-Macaulay local ring with a canonical ideal $\C$. Suppose that the type of $R$ is $2$. Then by Equation (\ref{eq1}), Proposition \ref{C2} and Remark \ref{Sallymult} we have that $s_0(S)= \cdeg(R)$, if $\rho(R)=2$, and $s_0(S)= \cdeg(R) + \sum_{j=2}^{\rho(R)-1} \lambda(\C^{j+1}/a\C^j)$ if $\rho(R)\geq 3$.}
\end{Remark}

It is shown in \cite[Theorem 3.16]{GMP11}
 that $R$ is an almost Gorenstein (but not Gorenstein) ring if and only if $s_0(S)=1$. Therefore it is natural to generalize as follows.

\begin{Definition}{\rm \cite[Definition 1.3]{CGKM} Let $R$ be a $1$-dimensional Cohen-Macaulay local ring with a canonical ideal $\C$. Then $R$ is a {\em $2$-almost Gorenstein local ring} ($2$-AGL for short), if $s_0(S)=2$; that is, $\e_1(\C)=\e_0(\C)-\lambda(R/\C)+2=\cdeg(R)+2$.}
\end{Definition}

\begin{Remark}\label{Gotonotation}{\rm We recall the notation of \cite[Section 2]{CGKM}, and we refer the reader there for details. Let $R$ be a Cohen-Macaulay local ring of dimension one with canonical ideal $\C$. Let $(a)$ be a minimal reduction of $\C$. Set ${\ds K=\dfrac{\C}{a}= \left\{\dfrac{x}{a} \; \vline \;  x\in \C \right\} }$  in the total ring of fractions of $R$. 
%Let $S=R[K]$. 
Then $K/R \simeq \C/(a)$.
%, so that $\cdeg(R)=\l(K/R)$. 
Therefore $\l(K/R)$ that appears in \cite{CGKM} is $\cdeg(R)$. }
\end{Remark}

Among many interesting results in \cite{CGKM}, Proposition 3.8 characterizes 2-AGL rings of type two. They are rings with minimal canonical degree and minimal canonical index. 
%We refer the reader to \cite[Section 2]{CGKM} for the notation used in their paper. 
We give here an alternative short proof using our set-up.
% and Proposition~\ref{C2}.
%We recall the notation of \cite[Section 2]{CGKM}, and we refer the reader there for details. Let $(a)$ be a minimal reduction of $\C$. Set $K=\dfrac{\C}{a}=\{\dfrac{x}{a} |\  x\in \C$\}. Let $S=R[K]$. Then $K/R \simeq \C/(a)$.

%Among many interesting results in \cite{CGKM} we have the following.

\begin{Proposition}\label{2AGLtype2}\cite[Proposition 3.8]{CGKM} Let $R$ be a $1$-dimensional Cohen-Macaulay local ring with a canonical ideal. Suppose that the type of $R$ is $2$. The following conditions are equivalent{\rm \,:}
\begin{enumerate}[{\rm (i)}]
\item $R$ is a $2$-AGL ring{\rm\,;}
\item $\cdeg(R)=2$ and $\rho(R)=2$. 
\end{enumerate}
\end{Proposition}

\begin{proof} 
We may assume that $\cdeg(R)\geq 2$, since $R$ is not almost Gorenstein. The conclusion follows from Remark \ref{relationships}.
%Assume that $R$ is a 2-AGL ring. We have $\cdeg(R)\geq 2$, since $R$ is not almost Gorenstein. Then Remark \ref{relationships} implies that $\cdeg(R)=2$ and $\rho(R)=2$. The converse follows again from Remark \ref{relationships}.
%Conversely, if  $\cdeg(R)=2$ and $\rho(R)=2$, Remark \ref{relationships} implies that $s_0(S)=2$.
\end{proof}

%\begin{Proposition}\cite[Proposition 3.7]{CGKM} Suppose that $r(R)=2$. The following conditions are equivalent:
%\begin{enumerate}[{\rm (1)}]
%\item $R$ is a 2-AGL ring;
%\item $\cdeg(R)=2$ and $S=K^2$. 
%\end{enumerate}
%\end{Proposition}

%\begin{proof} By Proposition~\ref{C2} we have that $\e_1(\C)= 2 \, \cdeg(R) + \sum_{j=2}^{\rho(R)-1} \lambda(\C^{j+1}/a\C^j)$.
%Assume that $R$ is a 2-AGL ring. We have $\cdeg(R)\geq 2$, since $R$ is not almost Gorenstein. Then $\cdeg(R)+2=\e_1(\C)= 2 \, \cdeg(R) + \sum_{j=2}^{\rho(R)-1} \lambda(\C^{j+1}/a\C^j)$ implies that $\cdeg(R)=2$ and $\rho(R)=2$. Conversely, if  $\cdeg(R)=2$ and $\rho(R)=2$, we have $\e_1(\C)=4=\cdeg(R)+2$, so $R$ is a 2-AGL ring.
%\end{proof}

In \cite[Section 6]{CGKM} the authors studied 2-AGL numerical semigroup rings. In particular, for 3 generated semigroup rings they obtained the following characterization, in the spirit of Corollary \ref{AGmatrix}.

\begin{Theorem}\label{2AGLabc} {\rm \cite[Theorem 6.4]{CGKM}} The ring $k \bl t^a,t^b,t^c \br$ is $2$-AGL if and only if after a suitable permutation of $a, b, c$ the Herzog matrix is ${\ds \left(\begin{matrix}
X^2&Y &Z\\ Y^{b_2}&Z^{c_2}&X^{a_2}\\ \end{matrix}\right)}$
with $a_2\geq 2$ and $b_2, c_2 >0$.
  \end{Theorem}
 
 \begin{Example}\label{ex2AGL}{\rm We consider some examples of rings of type 2.
 %, in accordance with Proposition \ref{2AGLtype2}.
 
 \begin{enumerate}[(1)]
 \item  Three generated semigroup rings with small multiplicity that are 2-AGL are characterized in \cite[Corollary 6.7]{CGKM}. One example is $R=k \bl t^3,t^7,t^8 \br$. 
 
\item  The ring $R=k \bl t^5,t^7,t^9 \br$ of Example \ref{579} has $\cdeg(R)=2$. After a permutation the Herzog matrix is ${\ds \left(\begin{matrix}
X^2&Y &Z\\ Y^{4}&Z&X\\ \end{matrix}\right) }$. 
The ring is not 2-AGL, according to Theorem \ref{2AGLabc} or \cite[Corollary 6.7]{CGKM}. We have that $\C=(x^2,y^4)$ is a canonical ideal of $R$ with minimal reduction $(x^2)$ and $\rho(R)=4$.
 
 \item  The ring $R$ of Example \ref{squares} has $\rho(R)=2$, but it is not 2-AGL, since $\cdeg(R)=8$. 
 \end{enumerate}}
 \end{Example}

To conclude this section, we would like to propose a possible further generalization.

\begin{Definition} {\rm Let $R$ be a $1$-dimensional Cohen-Macaulay local ring with a canonical ideal $\C$. Then $R$ is a {\em $n$-almost Gorenstein local ring} ($n$-AGL for short), if $s_0(S)=n$; that is, $\e_1(\C)=\e_0(\C)-\lambda(R/\C)+n=\cdeg(R)+n$.}
\end{Definition}

 \subsection*{3-AGL rings}\label{3AGLsubsection} It would be interesting to develop a theory of $3$-AGL rings, in particular for $3$-generated semigroup rings. We begin with the following characterization of rings of type 2, in the spirit of Proposition \ref{2AGLtype2}. We note that 3-AGL does not imply that $\cdeg(R)=3$.

%defining $n$-AGL rings as rings with $s_0(S)=n$. If $R$ is a 3-AGL ring ($\e_1(\C)=\cdeg(R)+3$) of type 2 we have the following. It would be interesting to have examples.
\begin{Proposition}\label{3AGL} Let $R$ be a $1$-dimensional Cohen-Macaulay local ring of type $2$ with a canonical ideal $\C$. Let $(a)$ be a minimal reduction of $\C$. Then $R$ is a $3$-AGL ring if and only if one of the following holds{\rm \,:}
\begin{enumerate}[{\rm (1)}]
\item $\cdeg(R)=3$ and $\rho(R)=2${\rm \,;}
\item $\cdeg(R)=2$, $\rho(R)=3$, and $\l(\C^3/a\C^2)=1$. 
\end{enumerate}
\end{Proposition}
\begin{proof}
We may assume that $\cdeg(R)\geq 2$, since $R$ is not almost Gorenstein. The conclusion follows from Remark \ref{relationships}.
%We have that $\cdeg(R)\geq 2$, since $R$ is not almost Gorenstein. By Proposition~\ref{C2} and the definition of 3-AGL we have that $\cdeg(R)+3=\e_1(\C)= 2 \, \cdeg(R) + \sum_{j=2}^{\rho(R)-1} \lambda(\C^{j+1}/a\C^j)$. If $\rho(R)=2$, then 
%By Theorem \ref{SallyofC} $\e_{1}(\C) = 2 \, \cdeg(R)$. 
%$\cdeg(R)+3=\e_1(\C)= 2 \, \cdeg(R)$ implies that $\cdeg(R)=3$.
%If $\rho(R)\geq 3$, the above equations yield $\cdeg(R) + \sum_{j=2}^{\rho(R)-1} \lambda(\C^{j+1}/a\C^j)=3$, which implies $\cdeg(R)=2$ and $\rho(R)=3$.
\end{proof}

\begin{Example}{\rm After a permutation $R=k \bl t^7,t^{12}, t^{15} \br$ has Herzog matrix ${\ds \left(\begin{matrix}
X^3&Y &Z\\
Y^{2}&Z^{2}&X^{3}\\
\end{matrix}\right)}$. 
Then $\C=(x^3,y^2)$ is a canonical ideal of $R$ with minimal reduction $(x^3)$ and $\rho(R)=2$. We have $\cdeg(R)=3$ by Theorem \ref{abc}. Therefore $R$ is $3$-AGL by Proposition \ref{3AGL}.} 
%$R=k[[t^7,t^{15}, t^{12}]]$
\end{Example}

\begin{Example}{\rm After a permutation $R=k \bl t^5,t^{9}, t^{11} \br$ has Herzog matrix ${\ds \left(\begin{matrix}
X^2&Y &Z\\
Y^{2}&Z^{3}&X\\
\end{matrix}\right)}$. 
Then $\C=(z,x)$ is a canonical ideal of $R$ with minimal reduction $(z)$. We have $\rho(R)=3$, $\l(\C^3/z\C^2)=1$. We also have $\cdeg(R)=2$ by Theorem \ref{abc}. Therefore $R$ is $3$-AGL by Proposition \ref{3AGL}.} 
\end{Example}
%$R=k[[t^9,t^{11}, t^{5}]]$

\begin{Example}{\rm The ring $R=k \bl t^5,t^{7}, t^{9} \br$ of Example \ref{ex2AGL} has $\cdeg(R)=2$, $\rho(R)=4$. Therefore $R$ is not $3$-AGL by Proposition \ref{3AGL}.} 
\end{Example}

In accordance with Theorem \ref{2AGLabc} and the above examples we ask the following.

\begin{Question} Suppose that after a suitable permutation of $a, b, c$, the ring $R=k \bl t^a,t^b,t^c \br$  has Herzog matrix  ${\ds \left(\begin{matrix}
X^3&Y &Z\\
Y^{b_2}&Z^{c_2}&X^{a_2}\\
\end{matrix}\right)}$
with $a_2\geq 3$ and $b_2, c_2 >0$.
Is $R$ a $3$-AGL ring?
\end{Question}

\begin{Question} Suppose that after a suitable permutation of $a, b, c$, the ring $R=k \bl t^a,t^b,t^c \br$  has Herzog matrix  ${\ds \left(\begin{matrix}
X^2&Y &Z\\
Y^{b_2}&Z^{c_2}&X\\
\end{matrix}\right)}$
with $b_2, c_2 >0$.
When is $R$ a $3$-AGL ring?
\end{Question}

\section{bi-canonical degree}\label{bicanonicalsection}

%Let $(R, \m)$ be a Cohen-Macaulay local ring that has a canonical ideal $\C$. 
The canonical degree is dependent on finding minimal reductions, which in general is a particularly hard task.
We consider in \cite{bideg} a degree that seems more amenable to computation. 

%We begin with the definition of bi-canonical degree in $1$-dimensional Cohen-Macaulay local rings. 

 \begin{Definition}
 {\rm Let $(R,\m)$ be a $1$-dimensional Cohen-Macaulay local ring with a canonical ideal $\C$.  The {\em bi-canonical degree of $R$} is the integer  $\ddeg(R)=\lambda(\C^{**}/\C)$, where ${\C}^{**}=\Hom_R(\Hom_R(\C,R),R)$ is the bidual of $\C$.}
%Then the integer $\cdeg(\RR)=\e_0(\C) -\l(\RR/\C)$ is independent of the canonical ideal $\C$.
 \end{Definition}
 
 \begin{Remark}{\rm We note that the bi-canonical degree is an invariant of the ring. We have $\ddeg(R)=0$ if and only if $\C$ is reflexive. Therefore, by \cite[Corollary 7.29]{HK2} in $1$-dimensional Cohen-Macaulay local rings the vanishing of $\ddeg(R)$ is equivalent to $R$ being Gorenstein.}
 \end{Remark}
 
 We define $\ddeg(R)$ in full generality as follows. We recall that $\deg(\tratto)$ is the multiplicity defined by the $\m$-adic filtration.
 
 \begin{Definition}\label{bideg}{\rm \cite[Theorem 3.1]{bideg}
Let $R$ be a Cohen-Macaulay local ring  of dimension $d \geq 1$  that has a canonical ideal $\C$. Then the {\em bi-canonical degree} of $R$ is the integer
  \[ \ddeg(R)= \deg(\C^{**}/\C) =
     \sum_{\tiny \h( \p)=1} \ddeg(R_{\p}) \deg(R/\p) = \sum_{\tiny \h(\p)=1} [\lambda(R_{\p}/\C_{\p}) - \lambda(R_{\p}/\C^{**}_{\p})] \deg(R/\p).
  \]   
}\end{Definition}

We show in \cite[Theorem 3.1]{bideg} that $\ddeg(R)$  is a well-defined finite sum independent of the canonical ideal $\C$. 

\begin{Remark}\label{bidegzero}{\rm 
%Let $(\RR, \m)$ be a Cohen-Macaulay local ring  of dimension $d \geq 1$ that has a canonical ideal $\C$.
$\ddeg(R)\geq 0$ and vanishes if and only if $R$ is Gorenstein in codimension $1$.}
% Furthermore, suppose that $\C$ is equimultiple. Then $\cdeg(\RR) =0$ if and only if $\RR$ is Gorenstein.
\end{Remark}

\begin{Remark}\label{trace}{\rm 
 Let $R$ be a Cohen-Macaulay local ring with a canonical ideal $\C$. 
 %Assume $R$ is Gorenstein on the punctured spectrum $\Spec(R)\setminus \{\m\}$.
 The {\em residue} of $R$  was introduced by Herzog, Hibi, and Stamate in \cite[Section 2]{HHS19}.
  A standard metric is ${\rm res}(R) =\deg(R/\tr(\C))$, where  $\tr(\C) = \C \cdot \C^{*}=
   \{f(x): f\in \C^{*}, x\in \C \}$ is the trace ideal of $\C$. 
  In dimension $1$,  we have $\ddeg(R)=\lambda(R/\tr(\C))={\rm res}(R)$  \cite[Proposition 2.3]{bideg}.
 They may differ in dimension $d>1$. }
 \end{Remark}

 \begin{Remark}{\rm In dimension $1$, $R$ is Gorenstein if and only if one of $\cdeg(R)$, $\ddeg(R)$ or ${\rm res}(R)$ vanishes, in which case all three vanish.
 If $R$ is not Gorenstein,  $\ddeg(R) \geq 1$, ${\rm res}(R)\geq 1$.  The cases when the minimal values are reached have the following designations.
 \begin{enumerate}[{\rm (1)}]
%\item  $\cdeg(R) = r(R) - 1$ if and only if $R$ is  an {\em almost Gorenstein ring}.
  \item  $\ddeg(R) = 1$ if and only if  $R$ is  a {\em Goto ring}.
  \item ${\rm res}(R)= 1$ if and only if $R$ is a {\em nearly Gorenstein ring}.
 \end{enumerate}
}\end{Remark}

We now recall very useful formulas for the computation of duals and biduals.
Let $\mathcal{Q}$ be the total ring of fractions of a Noetherian local ring $R$. If $I$ is an ideal containing a non zero divisor, then ${\ds  \Hom_R(I, R) = (R:_{\mathcal{Q}} I) }$.  A difficulty is that computer systems such as Macaulay2 \cite{Macaulay2} are set to calculate quotients of the form $(A:_R B)$ for ideals $A,B\subset R$, which is done with calculations of syzygies.  But it is possible to compute quotients in the total rings of fractions as follows.

\begin{Proposition}\cite[Proposition 3.3]{bideg} \label{bidual}
Let $I, J$ be $R$-ideals and let $a\in I$ be a regular element.

 \begin{enumerate}[{\rm (1)}]
  \item $\Hom_R(I,J)=a^{-1}(aJ:_{R} I)$. In particular,  $I^{*} = \Hom_R(I,R) = a^{-1} ((a) :_R I)$.
 \item  $I^{**} = \Hom_R(\Hom_R(I, R), R) =( (a):_R ((a):_R I))$.
  % \item $\tr(I) = I \cdot I^{*} = a^{-1} I  \cdot ((a):_\RR I)$.
\end{enumerate}
     \end{Proposition}

When $R$ is $1$-dimensional we discuss relationships between almost Gorenstein rings (i.e., rings with minimal value $\cdeg(R) = r(R) - 1$) and Goto rings. 

\begin{Theorem} {\rm (\cite[Theorem 4.2]{bideg}, \cite[Proposition 6.1]{HHS19})} \label{AGorddeg} Let $(R,\m)$ be a Cohen-Macaulay local ring of dimension $1$
 with a canonical ideal $\C$.  If $R$ is almost Gorenstein $($not Gorenstein$)$,  then $R$ is a Goto ring.
\end{Theorem}

\begin{proof} From the proof of Proposition \ref{1dimalmostg}
%Let $r$ be the type of $\RR$ and $(c)$ a minimal reduction of $\C$. Since $\RR$ is almost Gorenstein, 
there exists an exact sequence of $R$-modules
\[ 0 \rightarrow (a) \rightarrow \C \rightarrow X \rightarrow  0 \]
such that $(a)$ is a minimal reduction of $\C$ and $X \simeq (R/\m)^{r(R)-1}$.  By applying $\Hom_R(\cdot, (a))$ to the exact sequence above, we obtain the exact sequence
\[ 0\rar  \Hom_R(\C, (a)) \rar R \rar  \Ext_R^{1}(X, (a)). \]
%\rar \Ext^{1}(\C, (c)) \rar 0.\]
Since $R$ is not Gorenstein, we have $\C \neq (a)$. Therefore the image of $\Hom_R(\C, (a))=(a):_R \C$ in $R$ is a proper ideal of $R$. Since $\m \C\subset (a)$ it follows that $\m= (a):_R \C$.

By Proposition~\ref{bidual} we have that $\C^{**} =( (a):_R ((a):_R \C))= (a):_R \m$. Therefore $\C^{**}/(a)$ is isomorphic to the socle of $R/(a)$  in $R$ and we obtain 
\[\ddeg(R) = \lambda(\C^{**}/(a)) - \lambda(\C/(a)) =r(R)-(r(R)-1)= 1. \qedhere\] 
\end{proof}

\begin{Corollary}\label{one} Let $(R,\m)$ be a Cohen-Macaulay local ring of dimension $1$
 with a canonical ideal $\C$.  If $\cdeg(R)=1$,  then $\ddeg(R) =1$.
\end{Corollary}
\begin{proof} Since $\cdeg(R)=1$,  by Remark \ref{cdeg1} $R$ is almost Gorenstein (not Gorenstein). Therefore $\ddeg(R) =1$ by Theorem \ref{AGorddeg}.
\end{proof}

The example below shows that the
converse of Theorem \ref{AGorddeg} (and of Corollary \ref{one}) does not hold true. Therefore, in dimension one, Goto rings are a larger class than almost Gorenstein rings. 

\begin{Example}{\rm Let $R=k \bl t^5,t^7,t^9 \br$ as in Example \ref{579}. We have that $\cdeg(R)=2$, and $\ddeg(R)=1$, by Proposition \ref{bideg3generated}, or by a direct calculation (see \cite[Example 4.3]{bideg}).}
\end{Example}

\section{Comparison of canonical and bi-canonical degrees}\label{comparisonsection}

 In this section we discuss the Comparison Conjecture, proposed in \cite{bideg}.
 
  \begin{Conjecture}{\rm (Comparison Conjecture) \label{cdegvsfddeg}
  Let $(R, \m)$ be a Cohen-Macaulay local ring that has a canonical ideal $\C$. Then
$$\ddeg(R) \leq \cdeg(R).$$
}\end{Conjecture}

\begin{Remark}{\rm The Comparison Conjecture holds in the following cases:
 \begin{enumerate}[{\rm (1)}]
\item $R$ is Gorenstein in codimension one, since $\cdeg(R)=0$ and $\ddeg(R)=0$. 
\item $R$ is a Goto ring, since $\ddeg(R) =1$ implies $\cdeg(R)\geq 1$, otherwise $R$ would be Gorenstein in codimension one.
\item $R$ is an almost Gorenstein ring of dimension one, by Theorem \ref{AGorddeg}.
 \end{enumerate}}
\end{Remark}

\subsection*{The Comparison Conjecture for 3-generated numerical semigroup rings} We prove the conjecture for 3-generated numerical semigroup rings. Notation is as in Section \ref{sectionabc}. In \cite[Proposition 6.1]{bideg} we calculated $\ddeg(R)$ when $a_1\leq a_2$, $b_{2} \leq b_{1}$, and $c_{1} \leq c_{2}$. We obtained that $\ddeg(R)=  a_{1} b_{2} c_{1}$, which supported the Comparison Conjecture, since $\cdeg(R)$ is either $a_1 b_1 c_1$ or $a_2 b_2 c_2$, by Theorem \ref{abc}. Here we complete the calculation to obtain $\ddeg(R)$ in general.

\begin{Proposition}\label{bideg3generated}
Let $R= k \bl t^a, t^b, t^c \br$. Let $\varphi = \left(\begin{matrix}
X^{a_1} & Y^{b_1} & Z^{c_1}\\
Y^{b_2}&Z^{c_2}&X^{a_2}\\
\end{matrix}\right)$
 be the Herzog matrix such that $R \simeq k \bl X, Y, Z \br/I_{2}(\varphi)$. 
%Let $m_a=\min\{a_1,a_2\}$, $m_b=\min\{b_1,b_2\}$. 
Then $${\ds \ddeg(R)=  \min\{a_1,a_2\} *\min\{b_1,b_2\}* \min\{c_1,c_2\} }.$$
\end{Proposition}

\begin{proof} We note that there are essentially two (possibly overlapping) cases, after a suitable change of variables.
\begin{enumerate}[{\rm (1)}]
\item  Two among $\min\{a_1,a_2\}, \min\{b_1,b_2\}, \min\{c_1,c_2\}$ occur as exponents on the same column of the matrix.
\item All $\min\{a_1,a_2\}, \min\{b_1,b_2\}, \min\{c_1,c_2\}$ occur as exponents on the same row of the matrix.
\end{enumerate}

\medskip

Let   $x, y, z$ be the images of $X, Y, Z$ in $R$ respectively. Then we have
\[ x^{a_{1}+a_{2}} = y^{b_{2}}z^{c_{1}}, \quad  y^{b_{1}+b_{2}} = x^{a_{1}}z^{c_{2}}, \quad z^{c_{1}+c_{2}} = x^{a_{2}} y^{b_{1}}. \] 
The ideal  ${\ds \C = (x^{a_{1}}, y^{b_{2}})}$ is the canonical ideal of $R$ and  ${\ds (x^{a_{1}}) : \C = ( x^{a_{1}}, y^{b_{1}}, z^{c_{1}})}$, since 

\[ \begin{array}{ll}
{\ds y^{b_{1}} y^{b_{2}} = x^{a_{1}}z^{c_{2}} \in (x^{a_{1}}),} \quad & {\ds z^{c_{1}} y^{b_{2}} = x^{a_{1}+a_{2}}  \in (x^{a_{1}}).} 
\end{array}  \] 
Then by Proposition~\ref{bidual},  we have 
\[ \C^{**} = (x^{a_{1}}) : ( (x^{a_{1}}) : \C ) =  (x^{a_{1}}) : ( x^{a_{1}}, y^{b_{1}}, z^{c_{1}}). \] 
For case (1) we may assume that $a_1\leq a_2$, $b_{2} \leq b_{1}$, and $c_{1} \leq c_{2}$. Then we have
\[ \begin{array}{ll}
{\ds z^{c_{2}} y^{b_{1}} = z^{c_{1}}y^{b_{2}}z^{c_{2}-c_{1}}y^{b_{1}-b_{2}}   =   x^{a_{1}+ a_{2}} z^{c_{2}-c_{1}}y^{b_{1}-b_{2}}  \in (x^{a_{1}}),} & {\ds  z^{c_{2}}z^{c_{1}} = x^{a_{2}}y^{b_{1}} \in (x^{a_{1}}). } 
\end{array}  \] 
It follows that ${\ds \C^{**} = (x^{a_{1}}, y^{b_{2}}, z^{c_{2}})}$. Moreover, since $z^{c_{1}+c_{2}} \in (x^{a_{1}})$, we have that  $\C = (x^{a_{1}}, y^{b_{2}}) = (x^{a_{1}}, y^{b_{2}}, z^{c_{1}+c_{2}}) $. Therefore, 
\[ \begin{array}{rcl}
{\ds   \ddeg(R) }  &=& {\ds   \lambda(R/\C) - \lambda(R/\C^{**}) } \vspace{0.1 in} \\
&=& {\ds   \lambda( R/( x^{a_{1}}, y^{b_{2}}, z^{c_{1}+c_{2}}) ) - \lambda(R/(x^{a_{1}}, y^{b_{2}}, z^{c_{2}}) ) } \vspace{0.1 in} \\  &= & {\ds a_{1}b_{2}(c_{1}+c_{2}) - a_{1}b_{2}c_{2} } \vspace{0.1 in} \\ &=& {\ds  {a_1} {b_2}   {c_1}. } 
\end{array}   \]
For case (2) we may assume that that $a_1\leq a_2$, $b_{1} \leq b_{2}$, and $c_{1} \leq c_{2}$. 
We have that 
\[ \C^{**} = (x^{a_{1}}) : ( (x^{a_{1}}) : \C ) =  (x^{a_{1}}) : ( x^{a_{1}}, y^{b_{1}}, z^{c_{1}})=(x^{a_1}, y^{b_2}, z^{c_2}y^{b_2-b_1}), \] 
%since $z^{c_2}y^{b_2-b_1}y^{b_1}=z^{c_2}y^{b_2}=z^{c_2-c_1}y^{b_2}z^{c_1}=z^{c_2-c_1}x^{a_1+a_2} \in (x^{a_{1}})$, and \\
%$z^{c_2}y^{b_2-b_1}z^{c_1}=z^{c_1+c_2}y^{b_2-b_1}=x^{a_2}y^{b_1}y^{b_2-b_1}=x^{a_2}y^{b_2}\in (x^{a_{1}})$.
since 
\[ \begin{array}{l}
z^{c_2}y^{b_2-b_1}y^{b_1}=z^{c_2}y^{b_2}=z^{c_2-c_1}y^{b_2}z^{c_1}=z^{c_2-c_1}x^{a_1+a_2} \in (x^{a_{1}}),\\
\\
z^{c_2}y^{b_2-b_1}z^{c_1}=z^{c_1+c_2}y^{b_2-b_1}=x^{a_2}y^{b_1}y^{b_2-b_1}=x^{a_2}y^{b_2}\in (x^{a_{1}}).
\end{array}  \] 
Since $z^{c_{1}+c_{2}} = x^{a_{2}} y^{b_{1}}\in (x^{a_1})$, we have $\C=(x^{a_1}, y^{b_2}, z^{c_1+c_2})$, $\C^{**}=(x^{a_1}, y^{b_2}, z^{c_1+c_2},z^{c_2}y^{b_2-b_1})$. Because there are $a_1b_1c_1$ monomials of the form $x^iy^jz^k$, $0\leq i\leq a_1-1$, $b_2-b_1\leq j\leq b_2-1$, $c_2\leq i\leq c_1+c_2-1$, it follows that 

\[ \begin{array}{rcl}
{\ds   \ddeg(R) }  &=& {\ds  \lambda( R/( x^{a_{1}}, y^{b_{2}}, z^{c_{1}+c_{2}}) ) - \lambda(R/(x^{a_{1}}, y^{b_{2}}, z^{c_1+c_2},z^{c_2}y^{b_2-b_1}) } \vspace{0.1 in} \\  &= & {\ds a_{1}b_{2}(c_{1}+c_{2}) - \big(a_{1}b_{2}(c_{1}+c_2)-a_1b_1c_1 \big)} \vspace{0.1 in} \\
 &=& {\ds  {a_1} {b_1} {c_1}. } 
\end{array}   \]
\end{proof}
% in $\C^{**}\setminus \C$ we have the $a_1b_1c_1$ monomials of the form $x^iy^jz^k$, $0\leq i\leq a_1-1$, $b_2-b_1\leq j\leq b_2-1$, $c_2\leq i\leq c_1+c_2-1$. Therefore 
%$\ddeg(R)=\lambda( R/( x^{a_{1}}, y^{b_{2}}, z^{c_{1}+c_{2}}) ) - \lambda(R/(x^{a_{1}}, y^{b_{2}}, z^{c_1+c_2},z^{c_2}y^{b_2-b_1})=a_{1}b_{2}(c_{1}+c_{2}) - [a_{1}b_{2}(c_{1}+c_2)-a_1b_1c_1] =  {a_1} {b_1} {c_1}.$

From Theorem \ref{abc} and Proposition \ref{bideg3generated} we have the following.

\begin{Corollary} Let $R= k\bl t^a, t^b, t^c \br$ be a $3$-generated semigroup ring. Then $\ddeg(R) \leq \cdeg(R)$; that is, the Comparison Conjecture holds.
\end{Corollary}

\begin{Remark}
{\rm For $R=k \bl t^a, t^b, t^c \br$, the value of $\ddeg(R)$ was also independently
calculated in \cite[Proposition 3.1]{HHS21}, in accordance with the fact that in dimension one $\ddeg(R)$ is the residue of $R$ (Remark \ref{trace}). 
 }\end{Remark}

 \subsection*{Recent developments on the Comparison Conjecture for numerical semigroup rings}

 In the recent paper \cite{HK22}, Herzog and Kumashiro studied upper bounds on the colength of the trace of the canonical module in numerical semigroup rings, i.e., upper bounds on the bi-canonical degree. Let $R = K\bl H \br$. They study when the inequality $\l(R/ \tr(\C)) \leq g(H) - n(H)$
holds. In their set-up $g(H) = |\mathbb{N} \setminus H |$ and $n(H) = |\{h \in H : h < F(H)\}|$ denote
the number of gaps and the number of non-gaps, respectively, and $F(H)$
denotes the Frobenius number of $H$; that is, the largest integer in $\mathbb{N}\setminus H$. We refer to \cite{HK22} for the background. It follows from \cite[Lemma 2.2]{HK22} and Remark \ref{Gotonotation} that $g(H) - n(H)=\cdeg(R)$. Therefore the above inequality is equivalent to the Comparison Conjecture. We summarize their results. Recall that $r(R)$ denotes the type of $R$.
  
\begin{enumerate}[(1)]

\item Suppose that $r(R)\leq 3$. Then the Comparison Conjecture holds \cite[Theorem 2.4]{HK22}. See also Proposition 3.2 in the earlier paper \cite{HHS21} for the case of 3-generated semigroups. 

\item There is a counterexample to the Comparison Conjecture when $r(R)=5$ \cite[Example 2.5]{HK22}. Let $H = <13, 14, 15, 16, 17, 18, 21, 23>$. 
 Then $\ddeg(R)=9>\cdeg(R)=8$. The authors note that this example arises from far-flung Gorenstein rings, studied in \cite[Section 4]{HK22}. It is shown in \cite[Corollary 2.2]{HHS21} that the inequality $\l(R/ \tr(\C)) \leq n(H)$ always holds. $R$ is far-flung Gorenstein if and only if $\l(R/ \tr(\C))=n(H)$. Therefore, far-flung Gorenstein rings are good candidates for counterexamples, since the bi-canonical degree is as big as possible.
 
 \item The case $r(R)=4$ is open \cite[Question 2.6]{HK22}. The examples we tried so far support the conjecture.
 
 \end{enumerate}
 
 We conclude this section with the following question.
 
 \begin{Question}{\rm Can we find other classes of rings for which the Comparison Conjecture holds? What about 2-AGL rings?}
  \end{Question}
 
 \section{Canonical degrees and change of rings}\label{changeringsection}
 
 Degree formulas are often statements about change of rings. In \cite{blue1} and \cite{bideg} we examined several cases for cdeg and bi-deg respectively. In particular we considered polynomial extensions, power series and completions in \cite[Section 6]{blue1}. In this section we highlight and summarize the cases of augmented rings, $(\m:\m)$, and hyperplane sections. 
 
 \subsection*{Augmented rings} Let $(R,\m)$ be a $1$-dimensional Cohen-Macaulay local ring with a canonical ideal $\C$.
%Assume that $R$ is not a discrete valuation ring.
Consider the augmented ring $R \ltimes \m$.
That is, $R \ltimes \m = R\oplus \m \epsilon$, $\epsilon^2 = 0$. (Just to keep the components apart in computations we use $\epsilon$ as a place holder.)
From the results in \cite[Section 6]{blue1} we have that $(\C:_{R} \m) \times \C$ is a canonical ideal of $R \ltimes \m$. Motivated by \cite[Theorem 6.5]{GMP11} we obtained the following. Recall that $r(R)$ denotes the type of $R$.

\begin{Theorem} {\rm (\cite[Theorem 6.8, Corollary 6.9]{blue1}, \cite[Proposition 5.1]{bideg})}  Let $(R, \m)$ be a $1$-dimensional Cohen-Macaulay local ring with a canonical ideal. 
\begin{enumerate}[{\rm (1)}]
\item Suppose that $R$ is not a discrete valuation ring. Then \[ \cdeg(R \ltimes \m)= 2 \, \cdeg(R)+2 \quad \mbox{\rm and} \quad r(R \ltimes \m) = 2 \, r(R)+1.\]
\item $R$ is an almost Gorenstein ring if and only if $R \ltimes \m$ is an almost Gorenstein ring.
\item Suppose that $R$ is not Gorenstein. Then \[ \ddeg(R \ltimes \m) = 2\, \ddeg (R) -1.\]
In particular if $R$ is a Goto ring, then $R \ltimes \m$ is also a Goto ring.
\end{enumerate}
\end{Theorem}

Augmented rings in the 2-AGL context are discussed in \cite[Section 4]{CGKM}.

\subsection*{Calculations in $(\m:\m)$} Let $(R, \m)$ be a $1$-dimensional Cohen-Macaulay local ring with a canonical ideal $\C$. Let $\mathcal{Q}$ be the total ring of fractions of $R$ and $A  = (\m:_{\mathcal{Q}} \m)$. The Gorenstein property of $A$ was studied in \cite{GMP11}. In  \cite[Section 8]{bideg} we found a formula for $\cdeg(A)$ in relation to the invariants of $R$. We have that the ideal $\m \C$ is the canonical ideal of $A$.

\begin{Theorem}\label{mm}\cite[Theorem 8.3]{bideg}
Let $(R, \m)$ be a Cohen Macaulay local ring of dimension $1$ with a canonical ideal.  Let $\mathcal{Q}$ be the total ring of fractions of $R$ and $A  = (\m:_{\mathcal{Q}} \m)$.  Suppose that $R$ is not a discrete valuation ring and that
 $(A, \M)$  is a local ring. Let  $e=[A/ \M : R/\m]$. Then  
\[ \cdeg(A) = e^{-1} ( \cdeg(R) + \e_0(\m) - 2r(R) ).\] 
\end{Theorem}

\begin{Question}{\rm Under the same assumptions of Theorem \ref{mm} can we find a formula for $\ddeg(A)$?}
\end{Question}

\begin{Example}\cite[Example 8.8]{bideg} {\rm Let $R=k \bl t^5,t^7,t^9 \br$ as in Example \ref{579}. We have $\cdeg(R)=2$.  Moreover we have $\e_0(\m) = 5$,  $r(R)=2$ and $e=1$. Therefore $\cdeg(A)=3$. Let $\D$ be the canonical ideal of $A$. With a direct calculation we get $\lambda(A/\D) = 11$ and $\lambda(A/\D^{**})= 9$. Therefore
 $\ddeg(A) = 2$.}
\end{Example}

The structure of $(\m : \m)$ in connection with the 2-AGL property  is studied in \cite[Section 5]{CGKM}.

\subsection*{Hyperplane sections} 
 A change of rings issue is the comparison $\cdeg(R)$ to $\cdeg(R/(x))$ and $\ddeg(R)$ to $\ddeg(R/(x))$ for an
appropriate regular element $x$. We know that if $\C$ is a canonical module for $R$ then
$\C/x \C$ is a canonical module for $R/(x)$ with the same number of generators, so type is preserved under
specialization. However $\C/x \C$ may not be isomorphic to an ideal of $R/(x)$. Here is a case of good behavior. Suppose
$x$ is regular modulo $\C$. Then for the sequence
\[ 0 \rar \C \lar R \lar R/\C \rar 0,\]
we get the exact sequence
\[ 0 \rar \C/x\C  \lar R/(x) \lar R/(\C,x) \rar 0,\] so the canonical module
$\C/x\C$ embeds in $R/(x)$.
 Note that this leads to $\rho(\C) \geq \rho(\C/x\C)$.

 \medskip

\begin{Proposition}\label{hscdeg}\cite[Proposition 6.11]{blue1} Suppose $(R, \m)$ is a Cohen-Macaulay local ring of dimension $d \geq 2$ that has a canonical ideal $\C$. Suppose that $\C$ is equimultiple and  $x$ is regular modulo $\C$. Then
\[ \cdeg(R) \leq \cdeg(R/(x)).\]
\end{Proposition}

\begin{Question}\label{Bertinicdeg}{\rm
 An interesting question is when there exists $x$ such that $\cdeg(R) = \cdeg(R/(x))$ in $2$ cases: (i) $\C$ is equimultiple and (ii) $\C$ is not necessarily equimultiple. It would be a Bertini type theorem.
 }\end{Question}

For bideg$(R)$, in accordance with the Comparison Conjecture
we propose the following.

\begin{Conjecture} \cite[Conjecture 5.5]{bideg} {\rm If $\C$ is equimultiple and
$x$ is regular mod $\C$, then $ \ddeg(R) \geq \ddeg(R/(x))$.
}\end{Conjecture}

 \section{Generalization of bi-canonical degrees}\label{bicanonicalgeneralizationsection}
 
In our last paper with Vasconcelos \cite{BGHV} we defined the bi-canonical degree for rings where the canonical module is not necessarily an ideal. We also discussed generalizations to rings without canonical modules but admitting modules sharing some of their properties. In this section we summarize the main results and open questions.

\subsection*{bi-canonical degree of rings with a canonical module}

Let $(R, \m)$ be a Cohen-Macaulay local ring of dimension $d$ that has a canonical module $\omega$. For an $R$-module $E$, we denote the dual of $E$ by  ${\ds E^{*} = \Hom_{R}(E, R)}$.
By dualizing the finite free presentation of $\omega$
\[  F_{1} \stackrel{\varphi}{\lar} F_{0} \rar \omega \rar 0,\] 
we obtain the exact sequence
\[ 0 \rar \omega^{*} \rar F_{0}^{*} \stackrel{\varphi^{*}}{\lar} F_{1}^{*} \rar D(\omega) \rar 0,\]
where ${\ds D(\omega) = \coker(\varphi^{*})}$ is the Auslander dual of $\omega$. 
The module $D(\omega)$ depends on the chosen presentation but the values of ${\ds \Ext_{R}^i(D(\omega), R)}$, for $i\geq 1$, are independent of the presentation. By \cite[Proposition 2.6]{AusBr}, there exists an exact sequence 
\[ 0 \lar \Ext_{R}^1(D(\omega), R) \lar \omega  \lar  \omega^{**} \lar \Ext_{R}^2(D(\omega), R) \lar 0. \]

\begin{Definition}\label{bidegext}{\rm
Let $(R, \m)$ be a Cohen-Macaulay local ring with a canonical module $\omega$. The {\em bi-canonical degree} of $R$ is 
\[ \ddeg(R) = \deg( \Ext_{R}^1(D(\omega), R) ) + \deg( \Ext_{R}^2(D(\omega), R) ),\]
where ${\ds D(\omega) }$ is the Auslander dual of $\omega$ and $\deg(\tratto)$ is the multiplicity associated with the $\m$-adic filtration.
}\end{Definition}

An $R$-module $E$ is said to be {\em torsionless} if the natural homomorphism ${\ds \sigma: E \rar E^{**}}$ is injective. 
%Recall that a finite $R$-module $E$ is torsionless if and only if it is a submodule of a finite free module.  
We show that if $R$ has a canonical ideal, then Definition~ \ref{bideg} and Definition~\ref{bidegext} coincide.

\begin{Proposition}\label{torsionless}\cite[Proposition 3.2]{BGHV}
Let $(R, \m)$ be a Cohen-Macaulay local ring with a canonical module $\omega$. Then $\omega$ is torsionless if and only if $\omega$ is isomorphic to an ideal of $R$.
\end{Proposition}

Notice that by Definition~\ref{bidegext}, $\ddeg(R)=0$ if and only if $\omega$ is reflexive.  By \cite[Korollar 7.29]{HK2}, this is equivalent to $R_{\p}$ is Gorenstein for all primes $\p$ with height $1$. Therefore we can generalize Remark \ref{bidegzero}.

\begin{Proposition}\label{vanishingbideg} \cite[Proposition 3.3]{BGHV}
Let $(R, \m)$ be a Cohen-Macaulay local ring with a canonical module $\omega$. Then $\ddeg(R)=0$ if and only if $R$ is Gorenstein in codimension $1$.
\end{Proposition}

In accordance with Theorem \ref{AGorddeg}, we ask when the minimal value of bideg is attained.

\begin{Question}\label{Gotoring}{\rm 
Let $R$ be a Cohen-Macaulay local ring of dimension $1$ with a canonical module $\omega$. When is $\ddeg(R)=1$?
}\end{Question}

\begin{Remark}{\rm By Definition~\ref{bidegext}, $\ddeg(R)=1$ if and only if one of the following holds:
\begin{enumerate}[{\rm (1)}]
\item $ \Ext_{R}^1(D(\omega),R)=0$ and  $\deg( \Ext_{R}^2(D(\omega), R) )=1$;
\item $ \Ext_{R}^2(D(\omega),R)=0$ and  $\deg( \Ext_{R}^1(D(\omega), R) )=1$. 
\end{enumerate}
By Proposition \ref{torsionless} we have $ \Ext_{R}^1(D(\omega), R)=0$ if and only if $\omega$ is isomorphic to an ideal $\C$ of $R$. Therefore case (1) is equivalent to $\deg( \Ext_{R}^2(D(\omega), R) )=\l(\C^{**}/\C)=1$, i.e., $R$ is a Goto ring. It would be interesting to find an explicit example of case (2).
} \end{Remark}

The following proposition, which follows from classical results in \cite{BH}, is helpful in computing examples with Macaulay2, since it provides an explicit presentation of the Auslander dual as well as a criterion for the canonical module to be an ideal in terms of a matrix.

\begin{Proposition}\label{bidegcomp}\cite[Proposition 3.4]{BGHV}
Let $(S, \n)$ be a regular local ring. Let $I$ be an $S$-ideal of height $g$ such that $R=S/I$ is Cohen-Macaulay. Let 
\[ 0 \rar G_{g} \stackrel{\sigma}{\lar} G_{g-1} \rar \cdots \rar G_{1} \rar G_{0}=S \rar 0\]
be a minimal free $S$-resolution of $R$. Then:
\begin{enumerate}[{\rm (1)}]
\item ${\ds  \omega_{R} \simeq \coker(\sigma^{T} \otimes R)}$.
\item ${\ds  D(\omega_{R})  \simeq \coker( \sigma \otimes R) }$.
\item Let ${\ds \varphi= \sigma^{T} \otimes R }$ and ${\ds G_{g}= S^{p}}$. Then $\omega_{R}$ is an ideal if and only if ${\ds \h(I_{p-1}(\varphi) )\geq 1}$ and $I_{p}(\varphi)=0$.
\end{enumerate}
\end{Proposition}

Using Proposition~\ref{bidegcomp} we compute $\ddeg(R)$ in a ring when the canonical module is not an ideal.

\begin{Example}\label{ddegexample} \cite[Example 3.5]{BGHV} {\rm
Let $S$ be a regular local ring with the maximal ideal $(X, Y, Z)$. Let $R=S/I$ be given by the minimal resolution
\[ 0 \lar S^{2} \stackrel{\sigma}{\lar} S^{3} \lar S \lar R \lar 0, \; \mbox{where} \;  \sigma=\left(\begin{matrix}
X^{2} & Y^{2} \\ YZ & ZX \\ XZ & XY 
\end{matrix}\right).\]
Let $x, y, z$ be the images of $X, Y, Z$ in $R=S/I$ respectively. We obtain the free presentation of the canonical module $\omega_{R}$ of $R$:
\[ R^{3} \stackrel{\varphi}{\lar} R^{2} \rar \omega_{R} \rar 0,\]
where ${\ds \varphi = \sigma^{T} \otimes R =\left(\begin{matrix}
x^{2} & yz & xz \\ y^{2} & zx & xy \end{matrix}\right) }$.  Since  ${\ds \h(I_{1}(\varphi) ) =0}$, the canonical module $\omega_{R}$ is not an ideal.  Moreover, we have
\[ 0 \rar \omega_{R}^{*} \rar R^{2} \stackrel{\varphi^{*}}{\lar} R^{3} \rar D(\omega_{R}) \rar 0,\]
where ${\ds \varphi^{*} = \sigma \otimes R = \left(\begin{matrix}
x^{2} & y^{2} \\ yz & zx \\ xz & xy 
\end{matrix}\right)}$.  By using Macaulay2, we compute the bi-canonical degree:
\[ \ddeg(R) = \deg( \Ext_{R}^1(D(\omega_{R}), R) ) + \deg( \Ext_{R}^2(D(\omega_{R}), R) ) = 1+ 6 = 7. \qedhere \]
}\end{Example}

\subsection*{Precanonical ideals} We consider a Cohen-Macaulay local ring $R$ that does not necessarily have a canonical module. We look for rings that have ideals with properties similar to those of a canonical ideal. Ideals with similar set of conditions have appeared in the literature, notably among them closed ideals, studied by Brennan and Vasconcelos in \cite{BV1}, and spherical modules, introduced by Vasconcelos in \cite{V74}. We recall the necessary background.

\smallskip

An ideal $I$ is said to be {\em closed} if $\Hom_{R}(I, I)=R$. An $R$-module $E$ is said to be {\em spherical} if ${\ds \Hom_{R}(E, E)= R}$ and ${\ds \Ext^{i}_{R}(E, E)=0}$ for all $i \geq 1$ \cite[Section 4.3]{V74}. Spherical modules are also called {\em semidualizing} and have been studied extensively by Sather-Wagstaff. We refer to \cite{SWsemidualizing} for a survey. 

\smallskip

We note that the canonical module is semidualizing.
If $R$ is not Gorenstein, principal ideals generated by a non-zero divisor are semidualizing. Examples of rings with non trivial semidualizing ideals are constructed in \cite{SW-divisorclassgroup}.

\smallskip

We note an early result of Vasconcelos in this context.

\begin{Proposition}\cite[Proposition 4.9]{V74}. Let $(R, \m)$ be an Artinian local ring with $\m^2=0$. Let $E$ denote the the injective hull of $R/\m$. Let $P$ be a finitely generated $R$-module that satisfies ${\ds \Hom_{R}(P, P)= R}$ and ${\ds \Ext^{1}_{R}(P, P)=0}$. Then $P \simeq R$, or $P \simeq E$.
\end{Proposition}

In the set-up of the above proposition, an example of a module $P$ that is neither $R$ nor $E$ is in \cite[Section 4.3]{V74}. We now define a class of ideals in this setting.

\begin{Definition}\label{precanonicalideal}{\rm
Let $R$ be a Cohen-Macaulay local ring. An $R$-ideal $\P$ is called a {\em precanonical} ideal if $\Hom_{R}(\P, \P)=R$ and $\Ext_{R}^{1}(\P, \P)=0$. 
}\end{Definition}

\smallskip

In \cite[Theorem 4.2]{BGHV} we prove conditions under which a precanonical ideal is canonical. It would be interesting to find ideals that are precanonical, but not canonical. Next, we propose the following definition.

\begin{Definition}{\rm Let $R$ be a Cohen-Macaulay local ring with a precanonical ideal $\P$. The {\em bi-canonical degree} of $R$ {\em relative to} $\P$ is 
\[ \ddeg_{\P}(R) = \deg(\P^{**}/\P). \]
}
\end{Definition}

\begin{Question}\cite[Question 4.3]{BGHV} {\rm Which are the properties we can obtain from the bi-canonical degree of $R$ relative to $\P$  to understand the structures of the ideal and the ring?}
\end{Question}

Recall that if $\C$ is a canonical ideal of a $1$-dimensional Cohen-Macaulay local ring, then $\ddeg(R)=\lambda(\C^{**}/\C)=0$ (i.e., $\C$ is reflexive) if and only if $R$ is Gorenstein (i.e., $\C$ is isomorphic to $R$). In \cite[Proposition 4.5]{BGHV} we generalize as follows.

 \begin{Proposition}\label{resddeg3}
Let \(R\) be a \(1\)-dimensional Cohen-Macaulay local ring and \(I\) an ideal that contains a non-zero-divisor. If \(I\) is closed and reflexive,  then \(I\) is a principal ideal. 
\end{Proposition}

\begin{Corollary}Let \(R\) be a \(1\)-dimensional Cohen-Macaulay local ring with a precanonical ideal $\P$ that contains a non-zero-divisor. Then $\ddeg_{\P}(R)=0$ if and only if $\P$ is a principal ideal.
\end{Corollary}
\begin{proof} If $\ddeg_{\P}(R)=0$, then $\P$ is reflexive and so it is principal by Proposition \ref{resddeg3}. If $\P$ is a principal ideal the conclusion follows by Proposition \ref{bidual}.
\end{proof}

\begin{Question}{\rm
Can we remove the hypothesis that $I$ contains a non-zero-divisor in Proposition~\ref{resddeg3}?
}\end{Question}

\begin{Question}{\rm How does $\Ext_{R}^{1}(I,I)=0$ affect the structure of $I$?
}\end{Question}

We conclude this manuscript by mentioning that the questions on torsionless and reflexivity that came up while the authors worked on \cite{BGHV} inspired the paper \cite{EGGHIKMT} that studies torsionless and reflexivity in the more general context of $q$-torsionfree modules and for rings that are not necessarily Cohen-Macaulay. In particular, \cite[Proposition 2.3]{EGGHIKMT} generalizes Proposition \ref{torsionless}.

\end{document}